\documentclass[11pt]{article}

\usepackage{amsmath,amsthm,amsfonts,amssymb}

\usepackage{graphicx}

\newcommand{\captionfonts}{\footnotesize}
\makeatletter  
\long\def\@makecaption#1#2{%
  \vskip\abovecaptionskip
  \sbox\@tempboxa{{\captionfonts #1: #2}}%
  \ifdim \wd\@tempboxa >\hsize
    {\captionfonts #1: #2\par}
  \else
    \hbox to\hsize{\hfil\box\@tempboxa\hfil}%
  \fi
  \vskip\belowcaptionskip}
\makeatother   
\newcommand{\nwc}{\newcommand}

\newtheorem{prop}{Proposition}[section]
\newtheorem{lemma}[prop]{Lemma}
\newtheorem{rem}[prop]{Remark}

\nwc{\R}{\mathbb R}
\nwc{\Z}{\mathbb Z}
\nwc{\N}{\mathbb N}

\newcommand{\ignore}[1]{}

\nwc{\eps}{\varepsilon}
\nwc{\re}{Re\,}

\nwc{\wto}{\rightharpoonup}

\nwc{\ds}{\displaystyle}
\newcommand {\bedis} {\begin{displaymath}}
\newcommand {\edis} {\end{displaymath}}
\newcommand{\newbeqna} {\renewcommand {\arraystretch} {2}
                        \begin {displaymath} \begin {array}{crcl}}
\newcommand{\neweqna}{\end{array} \end {displaymath}}

\newcommand{\fbeqna}{\renewcommand {\arraystretch} {1.3}
\begin {displaymath}\begin{array}{rcll}}
\newcommand{\feqna}{\end{array}\end{displaymath}}
\newcommand {\beqna} {\begin{eqnarray*}}
\newcommand {\eqna} {\end{eqnarray*}}
\newcommand {\beqn} {\begin{eqnarray}}
\newcommand {\eqn} {\end{eqnarray}}

\begin{document}

\title{Oscillatory travelling wave solutions for coagulation
equations}

\author{
B. Niethammer
\thanks{Institute of Applied Mathematics,
University of Bonn, Endenicher Allee 60, 53115 Bonn, Germany}
\and
J. J. L. Vel\'{a}zquez\thanks{Institute of Applied Mathematics,
University of Bonn, Endenicher Allee 60, 53115 Bonn, Germany}
}

\maketitle
\begin{abstract}
We consider Smoluchowski's coagulation equation with kernels of homogeneity one of the form 
$K_{\varepsilon }(\xi,\eta) =\big( \xi^{1-\varepsilon
}+\eta^{1-\varepsilon }\big)\big ( \xi\eta\big) ^{\frac{\varepsilon }{2}}$.  Heuristically, in suitable exponential variables, one can argue that in this case the
long-time behaviour of solutions is similar to the inviscid Burgers equation and that for Riemann data solutions converge to a traveling wave for large times.
Numerical simulations in \cite{HNV16} indeed support this conjecture, but also reveal that the traveling waves are oscillatory and the oscillations become stronger
with smaller $\eps$. The goal of this paper is to construct such oscillatory traveling wave solutions and provide details of their shape via formal matched asymptotic expansions.
\end{abstract}



\section{Introduction}
\subsection{The coagulation equation}
The classical Smoluchowski coagulation equation is given by
\begin{equation}\label{smolu}
\partial_t f(t,\xi) = \tfrac 1 2 \int_0^{\xi} K(\xi{-}\eta,\eta) f(t,\xi{-}\eta) f(t,\eta)\,d\eta - f(t,\xi) \int_0^{\infty} K(\xi,\eta) f(t,\eta)\,d\eta\,,
\end{equation}
where $f(t,\xi)$ denotes the number density of clusters of size $\xi \in (0,\infty)$ at time $t>0$ and it describes the formation of larger clusters by binary coagulation of smaller
ones. The rate kernel $K$ is a nonnegative, symmetric function that contains all the information about the microscopic details of the coagulation process. 
For example, in \cite{Smolu16}
 Smoluchowski derived the kernel $K(\xi,\eta)=\big(\xi^{1/3}+\eta^{1/3}\big) \big( \xi^{-1/3} + \eta^{-1/3}\big)$ for clusters that move according to Brownian motion and merge if they 
 come close to each other. Various other kernels from different application areas can be found in the survey articles \cite{Friedlander00}, \cite{Aldous99} and \cite{Drake72}.

 Most kernels that one encounters in applications are homogeneous and we will assume this from now on. It is well-known that for kernels of homogeneity larger than one
 gelation occurs, that is the loss of mass at finite time, while for kernels of 
 homogeneity smaller than or equal to one, solutions conserve the mass if it is initially finite. 
A topic of particular interest in this latter case is whether the large-time behaviour is described by self-similar solutions. While this issue is well-understood for the constant
and the additive kernel \cite{MePe04}, for other kernels only few results are available. In the case of kernels with homogeneity strictly smaller than one,
existence results for self-similar solutions  have been established for a large class of kernels \cite{FouLau05}, \cite{EMR05}, while  uniqueness and convergence to self-similar
form has only recently been proved for some special cases \cite{NTV16b}, \cite{LauNiVel16}.

\subsection{Kernels with homogeneity one }

In this article we are going to consider kernels with homogeneity one. 
Van Dongen and Ernst \cite{vanDoErnst88} already noticed that one needs to distinguish two cases. 
In the first case, called class II kernels in \cite{vanDoErnst88},
one has $K(\xi,1) \to c_0>0$ as $\xi \to 0$. The most prominent example is the additive kernel $K(\xi,\eta)=\xi+\eta$, 
where it  is possible to solve the equation with the Laplace transform. In fact, for the additive kernel 
 there exists a whole family of self-similar solutions with
finite mass \cite{MePe04}. One of them has exponential decay, the others decay like a power law in a way such that the second moment is infinite. 
Recently, as the first mathematical result for non-solvable kernels with homogeneity one, existence of self-similar solutions for a range of class II kernels has been obtained in 
\cite{BNV16}.

On the other hand, kernels that  satisfy $\lim_{\xi \to 0} K(\xi,1)=0$,  are  called class-I kernels  in \cite{vanDoErnst88} and we also sometimes call them {\it diagonally dominant}.
In this case it is known that self-similar solutions with finite mass cannot exist, but  a suitable change of variables reveals that one can expect that the long-time behaviour
of solutions is similar as in the case of the inviscid Burgers equation. To explain the idea it is useful to rewrite \eqref{smolu} in conservative form, that is as
\begin{equation}\label{eq1}
 \partial_t \big( \xi f\big) + \partial_{\xi} \Big( \int_0^{\xi}  \int_{\xi-\eta}^{\infty} K(\eta,\rho) \eta f(\eta)f(\rho)\,d\rho\,d\eta\Big)=0\,.
 \end{equation}
Then we make the change of variables
 \begin{equation}\label{newvariables}
  \xi=e^{x} \qquad \mbox{ and } \qquad u(t,x)= \xi^2 f(t,\xi)
 \end{equation}
 such that  \eqref{eq1} becomes
 \begin{equation}\label{uequation}
  \partial_t u(t,x) = - \partial_x \Big( \int_{-\infty}^x\,dy\int_{x+ \ln (1-e^{y-x})}^{\infty} \,dz \,K(e^{y-z},1) u(t,y) u(t,z)\Big)\,.
  \end{equation}
  Notice also that $\int_{\infty} \xi f(t,\xi)\,d\xi = \int_{\R}u(t,x)\,dx$. Hence, mass conserving solutions to \eqref{smolu} correspond to  nonnegative integrable solutions
  of \eqref{uequation} with conserved $L^1$-norm.
  
As in \cite{HNV16}, we consider the rescaled function 
 $u_{\eps}(\tau,\tilde x) = 
 \frac{1}{\eps} u(\frac{\tau}{\eps^2},\frac{\tilde x}{\eps})$,  where $0<\eps\ll1$, to understand the large-time behaviour 
 of solutions to \eqref{uequation}. We find
 \begin{equation}\label{uepsequation}
  \begin{split}
 \partial_{\tau} u_{\eps}(\tau,\tilde x)& = - \partial_{\tilde x}\Big( \frac{1}{\eps^2} \int_{-\infty}^{\tilde x} \int_{\tilde x + \eps \ln \big( 1-e^{\frac{y-\tilde x}{\eps}}\big)}^{\infty}
  K\big(e^{\frac{y-z}{\eps}},1\big) u_{\eps}(\tau,y) u_{\eps}(\tau,z)\,dz\,dy\Big)\\
 &= - \partial_{\tilde x}
\Big( \frac{1}{\eps^2} \int_{-\infty}^{0} \int_{\eps \ln \big( 1-e^{\frac{y}{\eps}}\big)}^{\infty}
 K\big(e^{\frac{y-z}{\eps}},1\big) u_{\eps}(\tau,\tilde x + y) u_{\eps}(\tau,\tilde x+z)\,dz\,dy\Big)\\
  &\approx - c_0\partial_{\tilde x} \big(u_{\eps}(\tau,\tilde x)^2\big)\,
    \end{split}
 \end{equation}
with $c_0=\int_{-\infty}^0 \int_{\ln(1-e^y)}^{\infty} K\big(e^{y-z},1\big)\,dy\,dz<\infty $. Hence, we  conclude that $u_{\eps}$ approximately solves the Burgers equation.
Notice that in the second step in \eqref{uepsequation} we have used the translation invariance of the integral which is specific to the property that the kernel has homogeneity one.
Recall that for integrable nonnegative data with mass $M$ the solutions to the Burgers equation $\partial_tu + \partial_x \big(u^2\big)=0$ converge in the long-time limit to 
an $N$-wave with the same mass, i.e. $u(t,x) \sim \frac{1}{\sqrt{t}} N(x/\sqrt{t})$ with
$N(x)=\frac{x}{2}\chi_{[0,2 \sqrt{M}]}$. On the other hand, if one starts with data such that $u_0(x)\to b>0$ as $x \to -\infty$ the solution converges to a traveling wave with height $b$ and
speed related to $b$.

In order to investigate, whether the long-time behaviour of solutions to \eqref{uequation} is indeed the same as for the Burgers equation we performed in \cite{HNV16} numerical
simulations of \eqref{uequation}  for a family of kernels
\begin{equation}\label{kernelfamily}
  K_{\alpha}(\xi,\eta) = c_{\alpha} \big(\xi+\eta)^{1-2\alpha}\xi^{\alpha}\eta^{\alpha}\,, \qquad \alpha >0\,
   \end{equation}
   with a suitable normalization constant $c_{\alpha}>0$. 
Figure \ref{fig:Nwaves} shows the simulations for  several values of $\alpha$ for smooth initial data with compact support. 
  \begin{figure}[h!]
  \centering
  \includegraphics[width=.95\textwidth]{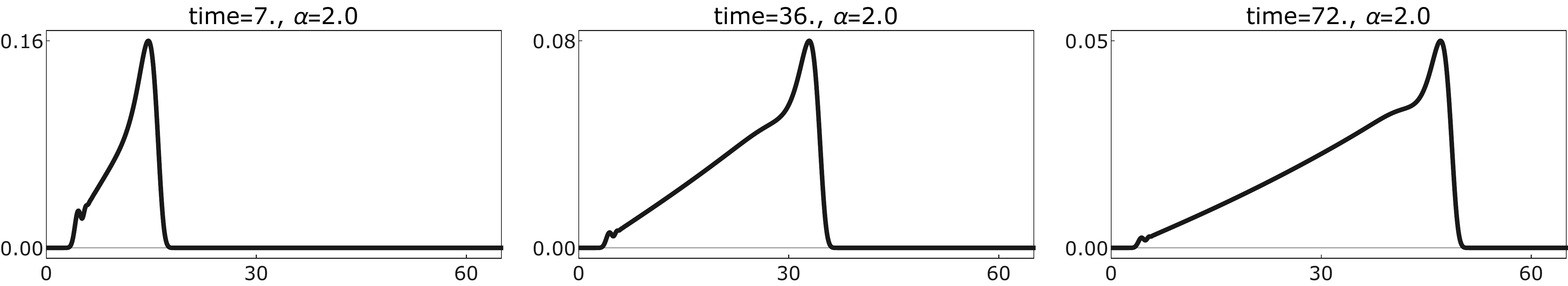}\\
  \includegraphics[width=.95\textwidth]{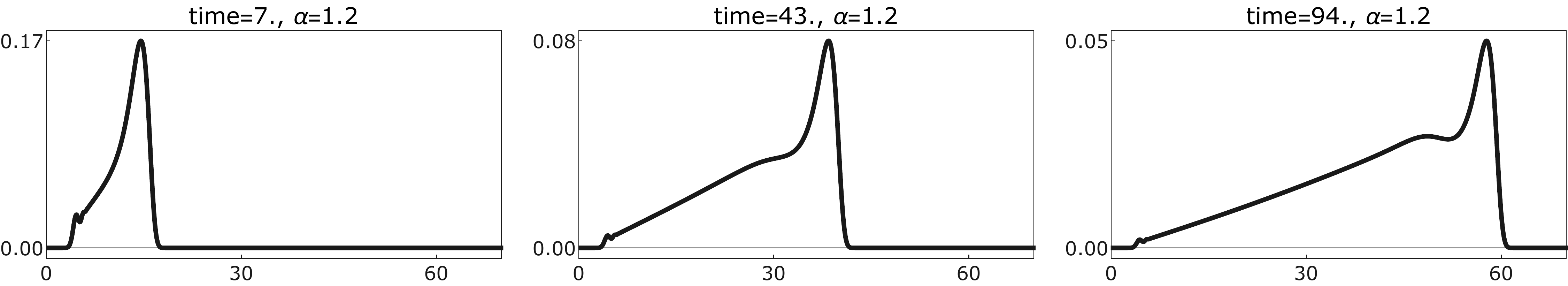}\\
  \includegraphics[width=.95\textwidth]{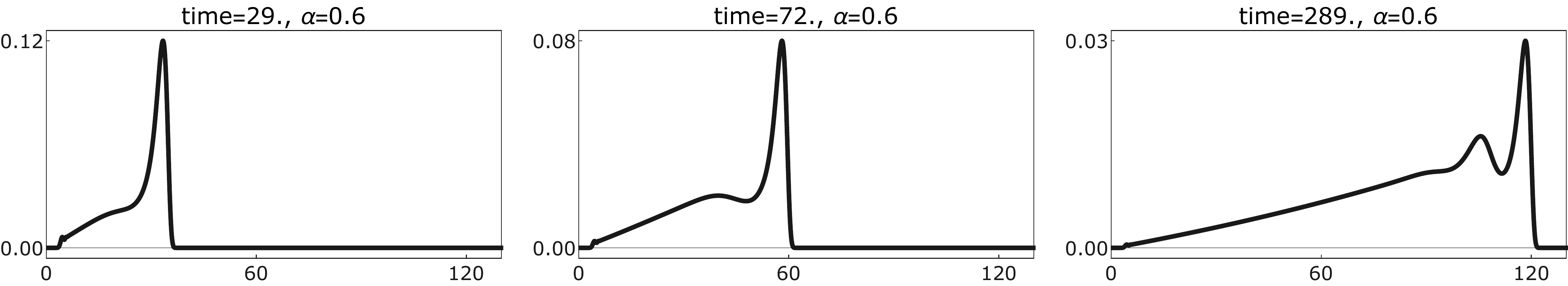}
  \caption{Convergence to the N-wave for $\alpha=2$ (top),
  $\alpha=1.2$ (middle) and $\alpha=0.6$ (bottom).}
  \label{fig:Nwaves}
  \label{fig:Nwaveneardiagonal}
\end{figure}
We see that solutions indeed converge to an $N$-wave in the long-time limit, but that the transition at the shock front is oscillatory, with oscillations becoming stronger when
$\alpha$ becomes smaller. One expects that the transition at the shock front is given by a rescaled traveling wave profile and simulations of \eqref{uequation} for Riemann data
(see Figure \ref{fig:twaves})
indeed  confirm that the traveling wave profiles are oscillatory for small $\alpha$. 

It is the purpose of this paper to establish the existence of such a traveling wave and compute details of its shape via formal matched asymptotic expansions. We will do this for kernels of the form
\begin{equation}\label{kernel}
K_{\varepsilon }(\xi,\eta) =\big( \xi^{1-\varepsilon
}+\eta^{1-\varepsilon }\big)\big ( \xi\eta\big) ^{\frac{\varepsilon }{2}}\,,
\end{equation}
where $\eps>0$ is a small parameter. We choose to consider kernels of the form \eqref{kernel} since the computations become slightly simpler. For small $\eps$ and $\alpha$ we do not
expect that the solutions for kernels as in \eqref{kernelfamily} and \eqref{kernel} respectively behave very differently.

 \begin{figure}[h!]
  \centering
  \includegraphics[width=.94\textwidth]{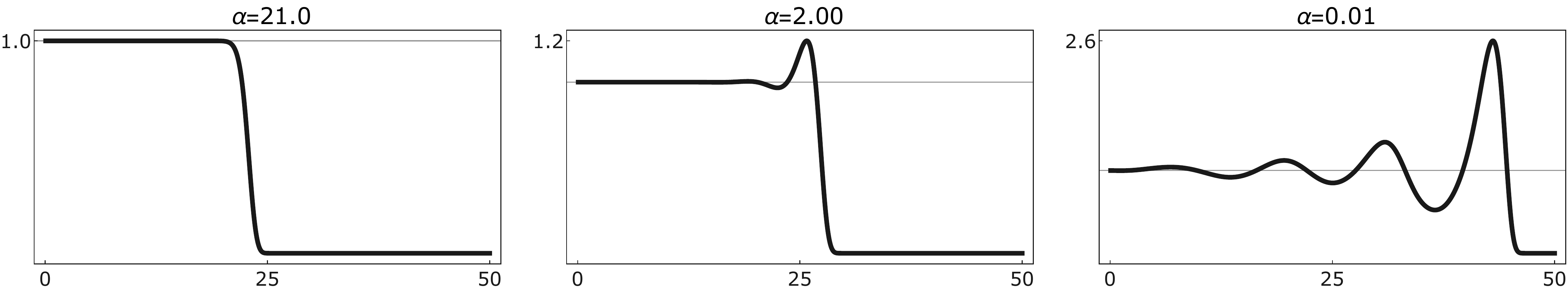}
  \caption{The shape of the traveling wave for different values of $\alpha$ ($u$ against $x$).}
  \label{fig:twaves}
\end{figure}
Before we proceed to an informal description of  construction of oscillatory traveling waves, we also mention  that we discuss in \cite{HNV16} in detail also  the situation for large $\alpha$. While for sufficiently large $\alpha$, the traveling wave appears to be monotone,
in this case  the constant solution is unstable which leads to additional phenomena; we refer to \cite{HNV16} for details.

\subsection{Overview of main result}\label{Ss.overviewmain}

The goal of this paper is to construct
 traveling wave solutions of \eqref{uequation}. The ansatz $u(t,x)=b G(x-bt)$ indeed leads to the equation
\begin{equation}
G(x) =\int_{-\infty }^{x}dy\int_{x+\ln ( 1-e^{y-x}) }^{\infty }dz K\big( e^{y-z},1\big) G(y)G(z)    \label{A1equation}
\end{equation}%
for the profile $G$ with the boundary condition $G(x) \to 0$ as $x \to \infty$. Such solutions $G(x)$ of
\eqref{A1equation} correspond to special solutions $f$ of \eqref{eq1} of the form $\xi^2 f(t,\xi)=b G(\ln \xi -b t)$, hence to traveling waves in the variable $\ln \xi$.

We recall that for $\eps=0$ in \eqref{kernel} we obtain the additive kernel, for which a whole family of mass conserving self-similar solutions  
 exist. One of them with exponential decay in the self-similar variable $\zeta=e^{-bt}\xi$, the other
ones with power law decay $\zeta^{-(2+\rho)}$ as $x \to \infty$ for $\rho \in (0,1)$ \cite{MePe04}.
In the variables \eqref{newvariables} this means that $G(x) \sim e^{-\rho x}$ as $x \to \infty$.
We also have that $G(x) \to 0$ as $x \to -\infty$ with the asymptotics $e^{-\frac{\rho}{1+\rho}x}$. 

In contrast, solutions to \eqref{A1equation} with diagonally dominant
kernel must satisfy $\lim_{x \to -\infty} G(x)=G_{\infty}>0$, hence, the integral of $G$ cannot be finite. This is the reason that self-similar solutions with finite mass cannot exist in
this case. Nevertheless,  nonnegative solutions to \eqref{A1equation} in the case of diagonally dominant kernels are relevant as well, since we expect that
they describe the long-time behaviour of solutions if one starts with Riemann data in the variable $x$, and they describe the transition at the shock front of the $N$-wave that appears
if the initial data are integrable. Notice also, that by varying the parameter $b$, we can change the step height of the traveling wave. For the following analysis it is convenient to
normalize $b=1$. The numerical simulations on the other hand were carried out by normalizing $G$ such that $G(x) \to 1$ as $x \to -\infty$.

Our goal in this paper is to study how  solutions $G$ to \eqref{A1equation} behave for kernels as in \eqref{kernel} in the regime $\eps \ll 1$.

We now describe in an informal way the main properties of the solution that we are going to construct. We try to obtain the solution essentially by a  shooting argument starting
at $x\to -\infty$. In Sections \ref{Ss.minusinfty}-\ref{Ss.linearization} we obtain that as $x \to -\infty$ a solution  $G$ of \eqref{A1equation} is oscillatory with the asymptotics
\begin{equation}
G(x)  \sim G_{\infty}+\mbox{Re}\Big(  Le^{i\varphi
}e^{\mu^{+}x}\Big)  +\Psi\big(x;L,\varphi\big)  +C_{\ast}e^{\mu_{\ast
}x} \qquad \mbox{ as } x \to -\infty.\ \label{GasymptModified}%
\end{equation}
with
\begin{equation*}
 G_{\infty} \approx \frac{\eps^2}{4}\,, \qquad \mu^+ \approx \frac{\eps^2}{8}+i \frac{\eps}{2}\, \qquad \mbox{ and }\qquad  \mu_{\ast} \approx 1-\frac{\eps}{2} \qquad \mbox{ as } \eps \to 0\,.
\end{equation*}

The constants $L>0, \varphi \in [0,2\pi)$ and $C_{\ast}\in \R$ can be
chosen freely, while the function $\Psi(X;L,\varphi)  $ yields
the contribution of the nonlinear terms induced by the term $\mbox{Re}%
( Le^{i\varphi}e^{\mu^{+}x})  $ which are larger than the
contribution due to the term $C_{\ast}e^{\mu_{\ast}x}.$
As we will explain in detail in the end of Section \ref{Ss.linearization} we can restrict ourselves, due to the translation invariance of \eqref{A1equation}, to $\varphi=0$
and $L \in [1,e^{\frac{\pi \eps}{2}})$. Then \eqref{GasymptModified} still describes a two-dimensional manifold. It turns out that the solution that we are seeking lies
close to the submanifold given by $C_*=0$. The reason is that  otherwise instabilities kick in, the solution becomes negative and is not a nonnegative solution of \eqref{A1equation} with the correct
behaviour as $x \to \infty$. 

 We remark here that our strategy is similar to our construction of self-similar solutions
to the coagulation equation with kernel $K(\xi,\eta)=(\xi,\eta)^{\lambda}$ with $\lambda \in (0,1/2)$ in \cite{MNV11}. 
Also in this case the solution develops oscillations that become more extreme with smaller $\lambda$. However, the details are somewhat different and in particular, there we do not have
the additional term  $C_{\ast}e^{\mu_{\ast}x}$ in \eqref{GasymptModified}, but instead can do the shooting in a one-dimensional manifold.

For simplicity of the argument assume  for the moment that the terms with $C_{\ast}$ can be neglected and that we can use $L$ as a single shooting parameter.
The main technical difficulty in the analysis of equation \eqref{A1equation} is to control the nonlocal integral terms.
A key idea in our approach is that as long as $G \ll 1$  we can approximate equation \eqref{A1equation} by a nonlinear ODE system that is a perturbation of a Lotka-Volterra system,
see Section \ref{Ss.lotkavolterra}.
We use an adiabatic approximation to compute the increase of an associated energy $E$ along the trajectories. This approximation is valid as along as $E \ll \frac{1}{\eps}$.
 When $E \sim \frac{1}{\eps}$ we enter what we call the intermediate regime, see Section \ref{Ss.intermediate}. In this regime $G$ develops on intervals of length of order
 $O(\ln \frac{1}{\eps})$ peaks of height one that can be approximately described
by self-similar solutions to the coagulation equation with additive kernel, see Section \ref{Ss.kinetic}. 
 These peaks are connected by  regions of length $O(\eps^{-2})$ 
in which $G$ is of order $O(e^{-\frac{1}{\eps}})$. 
In the regime where $G$ is small we can again approximate by simple ODE systems. Still, we have to distinguish three different regimes that
have to be coupled appropriately as described in Sections \ref{Ss.ode1}-\ref{Ss.ode2}. A key result in order to justify the approximation by ODEs is Lemma \ref{L.Gapprox} in 
Section \ref{Ss.Gapprox} that allows to control, if $G$ is small and in a certain sense regular,  the nonlocal terms by pointwise estimates.

For each $L \in [1,e^{\frac{\pi \eps}{2}})$ we can construct in this way a solution to \eqref{A1equation}. However, in general they become negative.      By a continuity argument,
explained in Section \ref{Ss.shooting},
we  show that there exists a parameter $L$ such that the corresponding solution satisfies $G(x)\to 0$ as $x \to  \infty$ with the asymptotics as given in \eqref{AppQFinal} of Lemma \ref{L.Glargex}.

In Section \ref{S.fast} we then address the additional parameter $C_{\ast}$. We are going to show that $C_{\ast}$ must be very small and carefully chosen in order to obtain a nonnegative
solution to \eqref{A1equation} that tends to zero as $x \to \infty$. More precisely, we consider the different regions explained above separately and show that if the term associated
with $C_{\ast}$ becomes of a size comparable to the solution of the Volterra-like problem, the corresponding solution   will not be admissible as a solution to the coagulation equation. 

Finally we emphasize that the construction of our solution is based on a matched asymptotic expansion. However, since we deal with an integral equation we can in general not
use standard
tools for asymptotic problems in ODEs, but have to estimate carefully the nonlocal terms corresponding to the integrals.


 \section{Construction of an oscillatory traveling wave.}\label{S.wave}

 \subsection{The behaviour for $x \to -\infty$}\label{Ss.minusinfty}

As a first step we derive the asymptotics of  a solution as $x \to -\infty$. For this we note 
\begin{equation*}
K_{\varepsilon }\big( e^{y},1\big) = 
\big( e^{(1-\varepsilon ) y}+1\big) e^{\frac{\varepsilon y}{2}%
}=e^{(1 -\frac{\varepsilon }{2}) y}+e^{\frac{\varepsilon y}{2}}
\end{equation*}
and it follows that
$\int_{-\infty }^{0}dy\int_{\ln ( 1-e^{y}) }^{\infty}K_{\varepsilon }\big(e^{ y-z},1\big) dz<\infty$.
We are interested in the asymptotics of the solution of (\ref{A1equation})
satisfying $\lim_{x \to -\infty}G(x)=G_{\infty}$ where
$G_{\infty }$ is given by
\begin{equation}\label{Ginftydef}
G_{\infty }\int_{-\infty }^{0}dy\int_{\ln ( 1-e^{y}) }^{\infty}dzK_{\varepsilon }\big( e^{y-z},1\big) =1\,.
\end{equation}
A simple explicit computation gives  $G_{\infty}=\frac{( 1-\frac{\varepsilon }{2}) \frac{\varepsilon }{2}}{\Gamma ( 1-\frac{\varepsilon }{2}) \Gamma ( \frac{\varepsilon }{2}) }$
and using the asymptotics $\Gamma(z)z \to 1$ as $z \to 0$ we also find
\begin{equation}\label{Ginftyasymptotics}
 G_{\infty} =\frac{\eps^2}{4} \big( 1 - \frac{\eps}{2} + O(\eps^2)\big) \qquad \mbox{ as } \eps \to 0\,.
\end{equation}

 \subsection{Reformulation as a Volterra-like problem}\label{Ss.reformulation}
 
 Suppose that $G(\cdot)  $ is a solution of \eqref{A1equation}. We 
rewrite this equation as
\begin{align}
G( x)   &  =\int_{-\infty}^{0}dy\int_{0}^{\infty}dz
\Big(e^{(1-\frac{\varepsilon}{2}) ( y-z)  }+e^{\frac{\varepsilon}{2}(y-z)  }\Big)  G(y+x)G(z+x)  \nonumber\\
& \qquad +\int_{-\infty}^{0}dy\int_{\ln(  1-e^{y})  }^{0}dz
\Big(e^{(  1-\frac{\varepsilon}{2}) (y-z)  }+e^{\frac{\varepsilon}{2}(y-z)  }\Big)  G(y+x)G(z+x) \label{GProd}%
\end{align}
and
 define the following functions
\begin{equation}
A(x)  =\frac{1}{2}\int_{-\infty}^{0}G(y+x)
e^{\frac{\varepsilon}{2}y}dy\,,\qquad \quad  B(x)  =\frac{1}{2}\int
_{0}^{\infty}e^{-\frac{\varepsilon}{2}z}G(z+x)  dz\ \label{Int1}%
\end{equation}%
\begin{equation}
P(x)  =\int_{-\infty}^{0}G(y+x)  e^{(1-\frac{\varepsilon}{2})  y}dy\,, \qquad \quad  Q(x)  =\int
_{0}^{\infty}e^{-(  1-\frac{\varepsilon}{2})  z}G(z+x)  dz\ \label{Int2}%
\end{equation}%
\[
J[ G] (x)  =\int_{-\infty}^{0}dY\int_{\ln(1-e^{y})  }^{0}dz\Big(  e^{(1-\frac{\varepsilon}{2})(y-z)  }+e^{\frac{\varepsilon}{2}(y-z)  }\Big)G(y+x)G(z+x)\,.
\]
Then
\begin{align}
\frac{dA(x)  }{dx}   =-\frac{\varepsilon}{2}A(x)
+\frac{1}{2}G(x) \,,& \qquad \quad  \frac{dB(x)  }{dx}%
=\frac{\varepsilon}{2}B(x)  -\frac{1}{2}G(x)
\label{ODE1}\\
\frac{dP(x)  }{dx}  =-\Big(  1-\frac{\varepsilon}{2}\Big)
P(x)  +G(x) \,,& \qquad \quad  \frac{dQ(x)}{dx}=\Big(  1-\frac{\varepsilon}{2}\Big)  Q(x)  -G(x) \label{ODE2}%
\end{align}
and rewrite (\ref{GProd}) as
\begin{equation}
G(x)  =4A(x)  B(x)  +P(x)
Q(x)  +J[G]  (x) \,.\label{Ident}%
\end{equation}

Notice that $J[G](x)  $ is determined by the
values of $G(y)  $ with $y\leq x$ and the
functions $A,B,P$ and $Q$ solve the ODEs in (\ref{ODE1}) and (\ref{ODE2}).
Hence, we will call  (\ref{ODE1})-(\ref{Ident})
a Volterra-like problem since the values of the functions $A,B,P,Q$ and $G$ at
$x$ depend only on the values of these functions for $y\leq x.$

A solution of (\ref{ODE1})-(\ref{Ident})
is however not necessarily a solution of \eqref{A1equation}.
More precisely, suppose that $G$ is a global solution of
(\ref{ODE1})-(\ref{Ident}). Then, the differential equations
for $B$ and $Q$ in (\ref{ODE1}) and (\ref{ODE2}) suggest that the functions
$B $ and $Q $  behave as
 $C_{B}e^{\frac{\varepsilon x}{2}}$ and $C_{Q}e^{(  1-\frac{\varepsilon}{2})  x}$ respectively as $x \to \infty$. However, the
formulas for $B(x)  $ and $Q(x)  $ in (\ref{Int1}) and
(\ref{Int2}) imply that these two functions are bounded for large $x$ if $G$
is bounded. This suggests that some suitable shooting parameters need to be
adjusted in order to obtain bounded solutions of (\ref{ODE1})-(\ref{Ident}).

\subsection{Linearization of the Volterra-like problem as $x\rightarrow
-\infty.$}\label{Ss.linearization}

We begin by constructing solutions of (\ref{ODE1})-(\ref{Ident})
 for $x\leq-x_0$ with  sufficiently large $x_0>0$. These solutions  depend on two parameters and satisfy
\begin{equation}\label{LimitG}
G(x)  \rightarrow G_{\infty} \sim \frac{\eps^2}{4}\text{ as }x\rightarrow-\infty
\end{equation}
where $G_{\infty}>0$ is as in \eqref{Ginftydef}. Moreover, we also assume 
that
\begin{equation}\label{Limits}
 \begin{split}
A(x)   & \rightarrow A_{\infty}=\frac{G_{\infty}}{\varepsilon
}\sim \frac{\eps}{4}\,, \qquad \quad
B(x)    \rightarrow B_{\infty}=\frac{G_{\infty}}{\varepsilon
}\sim \frac{\eps}{4}\\
P(x)   & \rightarrow P_{\infty}=\frac{2G_{\infty}}{2-\varepsilon  }\sim \frac{\eps^2}{4}\,,\qquad 
Q(x)   \rightarrow Q_{\infty}=\frac{2G_{\infty}}{2-\varepsilon  }\sim \frac{\eps^2}{4}
\end{split}
\end{equation}
as $x \to -\infty$. We now look for solutions that behave
asymptotically as
\[
(  G,A,B,P,Q)  =\big(  G_{\infty},A_{\infty},B_{\infty},P_{\infty
},Q_{\infty}\big)  +\Big(  G^1_{\infty},A^1_{\infty},B^1_{\infty},P^1_{\infty},Q^1_{\infty}\big)  e^{\mu x}\qquad \text{ as }x\rightarrow-\infty\,.
\]
We linearize in \eqref{ODE1}-\eqref{Ident}, to obtain after some rearrangements that
\begin{align*}
A^1_{\infty}  & =\frac{1}{2}\frac{G^1_{\infty}}{(\mu+\frac{\varepsilon}{2})  }\,,
\qquad B^1_{\infty}=\frac{1}{2}\frac{G^1_{\infty}}{\big(\frac{\varepsilon}{2}-\mu\big)  }\,, \qquad 
P^1_{\infty}  & =\frac{G^1_{\infty}}{\left(  \mu+1-\frac{\varepsilon}%
{2}\right)  }\,, \qquad  Q^1_{\infty}
=\frac{G^1_{\infty}}{\left(  1-\frac{\varepsilon}%
{2}-\mu\right)  }%
\end{align*}
and
\[
G^1_{\infty}   =\frac{2G_{\infty}}{\varepsilon}\frac{G^1_{\infty}}{\left(
\frac{\varepsilon}{2}-\mu\right)  }+\frac{2G_{\infty}}{\varepsilon}%
\frac{G^1_{\infty}}{\left(  \mu+\frac{\varepsilon}{2}\right)  }+\frac
{2G_{\infty}}{\left(  2-\varepsilon\right)  }\frac{G^1_{\infty}}{\left(
1-\frac{\varepsilon}{2}-\mu\right)  }+\frac{2G_{\infty}}{\left(
2-\varepsilon\right)  }\frac{G^1_{\infty}}{\left(  \mu+1-\frac{\varepsilon}%
{2}\right)  }
  +G_{\infty}G^1_{\infty} J(\mu,\eps)
\]
with 
\begin{equation}\label{Jdef}
J(\mu,\eps):=\int_{-\infty}^{0}dy\int_{\ln (1-e^{y})  }^{0}dz \Big(e^{(  1-\frac{\varepsilon}{2}) (y-z)  }%
+e^{\frac{\varepsilon}{2} (y-z)  }\Big) \Big(  e^{\mu y}+e^{\mu z}\Big)\,.
\end{equation}
We are looking for solutions with $\mbox{Re}(\mu) \geq 0$. Then $G_{\infty}^1\not=0$ since otherwise 
$\mbox{Re}(\mu)<0$ or 
$A^1_{\infty}=B^1_{\infty}=P^1_{\infty}=Q^1_{\infty}=0$.
Then, if $G^1_{\infty}\neq0$ we obtain after some lengthy but elementary rearrangements and integrations by parts, that 
\begin{equation}\label{roots1}
 \begin{split}
\frac{\Gamma\big(  1-\frac{\varepsilon}{2}\big)  \Gamma\big(\frac{\varepsilon}{2}\big)  \Gamma ( 1+\mu)  }{(1-\frac{\varepsilon}{2})  \frac{\varepsilon}{2}}  & =\Gamma\Big(1-\frac{\varepsilon}{2}+\mu\Big)  \Gamma\Big(  \frac{\varepsilon}%
{2}\Big)  \frac{(  1-\mu)  }{(  1-\frac{\varepsilon}{2})(  \frac{\varepsilon}{2}-\mu)  }\\
& \qquad  +\Gamma\Big(  \frac{\varepsilon}{2}+\mu\Big)  \Gamma\Big(1-\frac{\varepsilon}{2}\Big)  \frac{(1-\mu)  }{\frac
{\varepsilon}{2}\big(  1-\frac{\varepsilon}{2}-\mu\big)  }%
\end{split}
\end{equation}

We state here a result about the roots of \eqref{roots1}, whose proof is given in the Appendix.

\begin{lemma}\label{L.roots}
 Equation \eqref{roots1} has exactly three roots with $\mbox{Re}(\mu)\geq 0$, denoted by $\mu^{\pm}$ and $\mu_{\ast}$ respectively. 
 Their asymptotics are given by
\begin{equation}\label{muone}
\mu^{\pm}=\frac{\varepsilon^{2}}{8}\pm i \big( \frac{\varepsilon}{2}%
+\frac{\varepsilon^{2}}{8}\big) +O\big(  \varepsilon^{3}\big)
\qquad \mbox{ and } \qquad \mu_{\ast}\simeq1-\frac{\varepsilon}{2}\ \ \text{as\ \ }\varepsilon
\rightarrow0 \,.
\end{equation}
\end{lemma}

The linearization of the problem \eqref{ODE1}-\eqref{Ident} suggests that there exists
a three-dimensional family of solutions which indeed can be characterized by the asymptotics
\eqref{GasymptModified}
with free constants $L>0, \varphi \in [0,2\pi)$ and $C_{\ast}\in \R$.

For any choice of $\varphi, L$ and $C_*$ the solution of \eqref{ODE1}-\eqref{Ident}  is uniquely determined. In order to find a traveling wave 
solution to our original problem, we need to find
$\varphi, L$ and $C_*$ such that $B(x)\to 0$ and $Q(x)\to 0$ as $x \to \infty$.
However, due to the translation invariance
of \eqref{A1equation} we have  for any $a \in \R$ that
$ \tilde G_a(x):= G(x+a) $ gives the same solution up to translations, such that, denoting $\mu^+=\mu_1+i \mu_2$, 
\[
 \tilde G_a(x) \sim G_{\infty} + \mbox{Re} \Big( Le^{\mu_1 a} e^{i(\varphi + \mu_2 a)} e^{\mu_+ x}\Big) + \tilde \psi(x)+ C_*e^{\mu_*a} e^{\mu_*x} \qquad 
 \mbox{ as } x \to - \infty\,.
\]
Hence, we can identify solutions of \eqref{ODE1}-\eqref{Ident} up to translations with the set of real numbers contained between two consecutive intersections of the spiral
$\{ Le^{\mu^+a}\,|\,a \in \R\}$ with the real axis. If $L=1$, the next larger point on the real axis is $e^{\frac{2\pi\mu_1}{\mu_2}}$. 
Therefore we can assume that $\varphi=0$ and $L \in \big [1,e^{\frac{2\pi\mu_1}{\mu_2}}\big )\approx [1,e^{\frac{\pi \eps}{2}})$. Notice that if the value of 
$L$ is modified from $L$ to $L e^{\frac{2\pi\mu_1}{\mu_2}}$
the value of $C_*$ is modified to $C_*e^{\frac{2\pi\mu_*}{\mu_2}}$.

Thus, finding a nonnegative solution $G(x)$ to the original problem \eqref{A1equation} with $G(x) \to 0$ as $x \to \infty$ is equivalent to finding a solution to \eqref{ODE1}-\eqref{Ident}
with the asymptotics \eqref{GasymptModified} with $\varphi=0$, $L \in [1,e^{\frac{2\pi\mu_1}{\mu_2}})$ and $C_* \in \R$  such that $B(x), Q(x) \to 0$ as $x \to \infty$.

Notice that in the formula \eqref{GasymptModified} the second term changes on scales of order $\frac{1}{\eps}$,
 whereas the last term changes on scales of order one. This is a consequence of the fact that the characteristic length scales in the equations for $A$ and $B$
 in \eqref{ODE1} are different from the ones for $P$ and $Q$ in \eqref{ODE2}. In particular, the equation for the variable $Q$ indicates that $Q$ separates
 in scales of order one from $G$. This has the consequence  that $C_*$ must be very small, because otherwise the fast growth of $Q$ will produce a change of sign
 either for $Q$ or $B$ as we will see later. 
 Even though $C_*$ is very small it must be tuned in a  careful way because otherwise the above mentioned instability in \eqref{ODE2} will cause the solution to change sign.

 We will postpone the detailed analysis of the effect of the term $C_* e^{\mu_* x}$ to Section \ref{S.fast} and consider now solutions to \eqref{ODE1} and \eqref{ODE2}
 with $(1-\frac{\eps}{2})Q(x)\approx G(x)$. This reduces the two-dimensional shooting problem to a one-dimensional one with shooting parameter $L$.

\subsection{The Lotka-Volterra regime.}\label{Ss.lotkavolterra}

It is possible to derive an ODE approximation of (\ref{A1equation}) which is valid
as long as $G$ remains much
smaller than one. Since $G_{\infty }$ is of order $\varepsilon
^{2} $
we introduce the following rescaling
\begin{equation}\label{rescaling1}
G=G_{\infty}g,\quad  A=A_{\infty}a,\quad  B=B_{\infty}b,\quad  P=P_{\infty}p,\quad Q=Q_{\infty}q,\quad
x=\frac{u}{\varepsilon},\quad y=\frac{v}{\varepsilon},\quad z=\frac{w}{\varepsilon},
\end{equation}
where $A_{\infty},B_{\infty}, P_{\infty}$ and $Q_{\infty}$ are as in
(\ref{Limits}). Then (\ref{ODE1})-(\ref{Ident}) become
\begin{align}
\frac{da}{du}   =\frac{1}{2}(a+g) \,, \qquad \quad  &\frac{db}{du}=\frac{1}{2}(b-g) \label{FastODE1}\\
\frac{dp}{du}   =\Big(  \frac{1}{\varepsilon}-\frac{1}{2}\Big)
(-p+g) \,, \qquad \quad  & \frac{dq}{du}=\Big( \frac{1}{\varepsilon}-\frac{1}{2}\Big) (q-g) \label{FastODE2}%
\end{align}
and
\begin{equation}
g=\frac{4G_{\infty}}{\varepsilon^{2}}a b+\frac{4G_{\infty}}{(2-\varepsilon)^{2}}p q+\frac{G_{\infty}}{\varepsilon^{2}}%
\bar{J}[g] (u) \,, \label{LambdaIdent}%
\end{equation}
where
\begin{equation}
\bar{J}[g](u)  =\int_{-\infty}^{0}%
dv\int_{\varepsilon\ln(  1-e^{\frac{v}{\varepsilon}})  }%
^{0}dw\Big(  e^{(  \frac{1}{\varepsilon}-\frac{1}{2}) (v-w)  }+e^{\frac{1}{2}(v-w)  }\Big)
g(v+u)  g(w+u)
\label{JKernel}%
\end{equation}

\begin{lemma}\label{L.odeapproximation}
Under the assumption that $g \to 1$ as $u \to -\infty$ and 
\begin{equation}\label{assumption1}
 |a(u)-1| + |b(u)-1| + |p(u)-1| + |q(u)-1| \leq C \|g -1 \|_{u}\,, \qquad   \Big |\frac{dg}{du}\Big| \leq C \Vert g -1 \Vert_{u}
\end{equation}
where $\|f \|_{u}:=\sup_{s \leq u} |f(s)|$,
we can rewrite \eqref{LambdaIdent} as the system of ODEs \eqref{FastODE1} together with
\begin{equation}
g ( u ) =a( u ) b( u ) -\frac{\varepsilon }{2}\big( a( u ) b( u ) -1\big) +\frac{\varepsilon }{2}\big(g ( u )
a(u)  -1\big) + R(u)\ \label{lambdaapprox}
\end{equation}
with $|R(u)| \leq C \eps^2 \big(1+ \Vert g  \Vert_{u} \big)\Vert g -1\Vert_{u}$.
\end{lemma}
\begin{proof}
To approximate $g$ we use \eqref{Int1} 
and \eqref{Ginftydef} in \eqref{LambdaIdent} to find
\begin{align*}
g (u) =&a(u)  b(u) 
+ \Big( \frac{4 G_{\infty}}{\eps^2}-1\Big) \big(a(u)b(u)-1\big) + \frac{4G_{\infty}}{(2-\eps)^2} \big(p(u) q(u) -1 \big) \\
&+\frac{G_{\infty }}{\varepsilon ^{2}}\int_{-\infty }^{0}dv
\int_{\varepsilon \ln ( 1-e^{\frac{v }{\varepsilon }})
}^{0}dw \Big( e^{( \frac{1}{\varepsilon }-\frac{1}{2})( v -w) }+e^{\frac{1}{2}( v -w) }%
\Big) \big( g ( v +u ) g ( w +u ) -1\big)\\
&=: a(u) b(u)+ (I)+(II)  + (III)_a + (III)_b\,.
\end{align*}
We first note that
\begin{equation}\label{lambdaest}
  \big| g ( v +u ) g ( w +u ) -1\big|\leq C \big(1+ \Vert g \Vert_{u} \big)\Vert g -1\Vert_{u}\,.
\end{equation}
With the assumption \eqref{assumption1} we have
$|(II)|
\leq C\varepsilon ^{2}\big(1+ \Vert g  \Vert_{u} \big)\Vert g -1\Vert_{u}$.
Furthermore, recalling  \eqref{Ginftyasymptotics}, we have
$\frac{G_{\infty }}{\varepsilon^{2}}-\frac{1}{4} =-\frac{\varepsilon }{8}
+( \varepsilon ^{2})$, whence
as a consequence we find 
\begin{align*}
(I)
&=-\frac{\varepsilon }{2}\big( a( u ) b( u ) -1\big) +O\big( \varepsilon ^{2}\Vert g -1\Vert_{u}\big(1+ \Vert g  \Vert_{u} \big) \big)\,.
\end{align*}
In addition we can estimate
\begin{align*}
|(III)_b|
&\leq C\big(1+ \Vert g  \Vert_{u} \big)\Vert g -1\Vert_{u} \int_{-\infty }^{0}dv
\int_{\varepsilon \ln ( 1-e^{\frac{v }{\varepsilon }})}^{0}dw e^{\frac{1}{2}(v -w) } \\
&
=C\varepsilon ^{2}\big(1+ \Vert g  \Vert_{u} \big)\Vert g -1\Vert_{u} \int_{-\infty
}^{0}dv \int_{\ln ( 1-e^{v }) }^{0}dw e^{\frac{%
\varepsilon }{2}(v -w ) }\,
\\ & \leq C\varepsilon ^{2}\big(1+ \Vert g  \Vert_{u} \big)\Vert g -1\Vert_{u}
\end{align*}
since
$ \int_{-\infty
}^{0}dv \int_{\ln ( 1{-}e^{v }) }^{0}dw e^{\frac{%
\varepsilon }{2}(v -w) } = \int_0^1 dy y^{\frac{\eps}{2}-1} \int_{1-y}^1 dz  z^{-(1+\frac{\eps}{2})} \leq C$.

We finally claim that
\begin{equation}\label{3aclaim}
 (III)_a = \frac{\eps}{2}\big(g a -1\big) + O\Big(\eps^2\big(1+ \Vert g  \Vert_{u} \big)\Vert g -1\Vert_{u}\Big )\,.
\end{equation}
Indeed, after changing the order of integration we obtain
\begin{equation*}
(III)_a=\frac{G_{\infty }}{\varepsilon ^{2}}\int_{-\infty }^{0}dw
\int_{\varepsilon \ln ( 1-e^{\frac{w }{\varepsilon }})
}^{0}dv e^{( \frac{1}{\varepsilon }-\frac{1}{2})(v-w }\big( g ( v +u ) g ( w
+u ) -1\big)\,.
\end{equation*}
In the region $w \geq -\eps$ we use \eqref{lambdaest} and the fact that $\int_{-\eps }^{0}dw
\int_{\varepsilon \ln( 1-e^{\frac{w }{\varepsilon }})
}^{0}dv e^{( \frac{1}{\varepsilon }-\frac{1}{2})( v-w) } \leq C\eps^2$ which follows from scaling. Therefore we can replace the
integral $\int_{-\infty}^0dw$ by $\int_{-\infty}^{-\eps} dw$  introducing an error
of order $C\eps^2 \big(1+ \Vert g  \Vert_{u} \big)\Vert g -1\Vert_{u}$.

Furthermore, in the remaining integral, we can,  due to \eqref{Ginftyasymptotics},  replace the prefactor $\frac{G_{\infty}}{\eps^2}$ by $\frac 1 4$ and this yields
 an error of order $C\eps^2\big(1+ \Vert g  \Vert_{u} \big)\Vert g -1\Vert_{u} $
since $\int_{-\infty}^{-\eps} dw \int^{0} _{\eps \ln (1-e^{w/\eps})} dv e^{( \frac{1}{\varepsilon }-\frac{1}{2})(v-w) }\leq C \eps$. 

In the remaining integral we can also replace $e^{(\frac{1}{\eps} - \frac 1 2 )(v -w)}$ by $e^{-(\frac{1}{\eps} - \frac 1 2 )w}$, using $-\eps \leq v \leq 0$, Taylor expansion 
and the fact that $\int_{-\infty}^{-\eps} \int_{\eps \ln (1-e^{w/\eps})}^0 |v| e^{-(\frac{1}{\eps} - \frac 1 2 )w }\leq C \eps^3$. 
 This gives an additional error of the order
$C\eps^3\big(1+ \Vert g  \Vert_{u} \big)\Vert g -1\Vert_{u} $.
Finally, due to \eqref{assumption1} we have $|g(u+v)-g(u)| \leq C \|g -1\|_{u} |v|$, such that we can replace $g(u+v)$ by $g(u)$ introducing
another error of the order $C\eps^3\big(1+ \Vert g  \Vert_{u} \big)\Vert g -1\Vert_{u} $.
Thus, we have found
\[
 (III)_a = \frac{1}{4} \int_{-\infty}^{-\eps} dw \int^0_{\eps \ln (1-e^{w/\eps})} dv e^{-(\frac{1}{\eps} - \frac 1 2) w} (g(u) g(w + u)-1) 
 + O(\eps^2 \big(1+ \Vert g  \Vert_{u} \big)\Vert g -1\Vert_{u}\big).
\]
We can now integrate over $v$ and use $-e^{-\frac{w}{\eps}} \ln (1-e^{\frac{w}{\eps}}) = 1+ O(e^{\frac{w}{\eps}})$ for $w \leq - \eps$ to find
\[
 \begin{split}
  (III)_a&=  \frac{\eps}{4} \int_{-\infty}^{-\eps} dw e^{\frac{w}{2}} \big( g (u) g (u + w)-1\big) 
  + O(\eps^2 \big(1+ \Vert g  \Vert_{u} \big)\Vert g -1\Vert_{u}\big)\\
  &=  \frac{\eps}{4} \int_{-\infty}^{0} dw e^{\frac{w}{2}} \big( g (u) g (u + w)-1\big) 
  + O(\eps^2 \big(1+ \Vert g  \Vert_{u} \big)\Vert g -1\Vert_{u}\big)\,.
 \end{split}
\]
Since $g \to 1$ as $u \to -\infty$ and since $a$ is bounded, we find $a(u) = \frac{1}{2} \int_{-\infty}^0 dw e^{\frac{w}{2}} g (u + w)$ and the claim \eqref{3aclaim}
follows from the previous formula.
\end{proof}

Lemma \ref{L.odeapproximation} suggests to approximate problem \eqref{FastODE1}-\eqref{LambdaIdent} by the system of ODEs
\begin{align}
\frac{da}{du} &  =\frac{1}{2}\Big (  -a+ab - \frac{\eps}{2}ab(1-a)\Big)\label{aode}\\
\frac{db}{du}&=\frac{1}{2}\Big(  b- ab +\frac{\eps}{2}ab(1-a)\Big) \label{bode}\\
\frac{dp}{du} &  =\Big(  \frac{1}{\varepsilon}-\frac{1}{2}\Big) 
\Big(  -p+ ab - \frac{\eps}{2}ab(1-a)\Big) \label{ellode} \\
\frac{dq}{du}&=\Big(  \frac
{1}{\varepsilon}-\frac{1}{2}\Big) \Big(  q-ab + \frac{\eps}{2} ab(1-a)\Big) \label{node}%
\end{align}
Linearizing this system around the solution $a=b=p=q=1$ we find a three-dimensional unstable manifold characterized by the eigenvalues 
$\pm \frac{i}{2} + \eps \big( \frac{1}{8} \pm \frac{i}{2}\big)$ and $\big( \frac{1}{\eps}-\frac{1}{2}\big)$. Recalling the rescaling \eqref{rescaling1}, this agrees with the results 
of Lemma \ref{L.roots}.

Notice that the third eigenvalue is much larger than the other two, which is  related to the fact, that the solutions to \eqref{ellode} and \eqref{node} change
on a faster  scale. As discussed before, we restrict  ourselves for the moment to solutions for which $p=g + O(\eps)$ and $q = g + O(\eps)$.
In particular this implies, that $p$ and $q$ satisfy the assumptions of Lemma \ref{L.odeapproximation} if $a$ and $b$ do so. Therefore we consider  for the moment
solutions of \eqref{aode} and \eqref{bode}.

\subsubsection{Adiabatic increase of the amplitude.}\label{Sss.adiabatic}

We have seen that with increasing $u$ the trajectories follow almost elliptic curves, but spiral outwards in each round. This property does not only hold
for $a$ and $b$ close to one, but also when $|a-1|+|b-1|=O(1)$. 

We compute the increase of the amplitude of the spiral
by using an adiabatic approximation.
More precisely, we recall that 
the leading order of the system  \eqref{aode}-\eqref{bode} is 
\begin{equation}\label{leadingorderode}
\frac{da }{du } =-\frac{a }{2}+\frac{a b }{2} = \frac{a}{2}(b-1) \,,\qquad 
\frac{db  }{du } =\frac{b }{2}-\frac{ab}{2} = \frac{b}{2}(1-a)
\end{equation}
and the energy
\begin{equation}
 \label{energydef}
 E=a+b-\ln(ab) -2
\end{equation}
is conserved by \eqref{leadingorderode}.

We compute now the
change of energy for the perturbed problem. We have
$\frac{dE}{du } 
=-\Big( \frac{1-a}{a}\Big) \frac{da}{du }+\Big( \frac{1-b}{b}\Big) 
\frac{db}{du }$ and
using \eqref{aode} and \eqref{bode}
we obtain
\[\frac{dE}{du }=\frac{\varepsilon }{4}( b-a)(1-a)\,.\]

We need to estimate the change of
energy in a cycle where we use the approximation \eqref{leadingorderode}.
The change of the energy in a period is given by $\frac{\varepsilon }{4}D(E)$
where
\begin{equation}\label{changeofenergy}
D(E) =\int_{0}^{T(E) }( b-a)(1-a) du 
=2\int_{C(E) }(b-a) \frac{db}{b}\,,
\end{equation}
where we used  that $ (1-a) du =\frac{2}{b}db$
and where $C( E) $ denotes the contour in the plane $(a,b)$ associated to the trajectory with energy $E$.
We can parameterize it by  two curves \linebreak 
$\{ ( a_{+}(b) ,b) :b\in (b_{\max }( E) ,b_{\min }( E) )\}\cup
\{ ( a_{-}( b) ,b) :b\in( b_{\min }(E) ,b_{\max }( E) )\} .$ Then
\[
D( E) =-2\int_{b_{\min }(E) }^{b_{\max }(E) }(b-a_{+}(b)) \frac{db}{b}+2\int_{b_{\min
}( E) }^{b_{\max }( E) }(b-a_{-}(b) ) \frac{db}{b} 
=2\int_{b_{\min }(E) }^{b_{\max }(E) }( a_{+}( b) -a_{-}( b) ) \frac{db}{b}
\]
Since $a_{+}( b) >a_{-}( b) $ for any $b\in (b_{\min }(E) ,b_{\max }(E) )$ we find $D(E)>0$ and the energy increases.
We now compute  the
asymptotics of the solutions as $E\rightarrow 0$ and $E\rightarrow \infty$.
We recall that the contours $C(E) $ are defined by \eqref{energydef}.

We first consider the limit  $E \rightarrow 0$. In this regime 
 we obtain  the asymptotics
$a_{+}( b) \sim 1+\sqrt{2E -( b-1) ^{2}}$ and 
$a_{-}( b)  \sim 1-\sqrt{2E -(b-1)^{2}}$, 
as well as
$b_{\min }(E)  \sim 1-\sqrt{2E}$ and 
$b_{\max }(E)  \sim 1+\sqrt{2E }$. Furthermore we obtain 
$a_{+}(b) -a_{-}(b)  \sim 2\sqrt{2E -( b-1)^{2}}$. 
Therefore
\begin{align*}
D(E)  &=2\int_{b_{\min }(E) }^{b_{\max }(E) }( a_{+}( b) -a_{-}( b) ) \frac{db%
}{b}\sim 4\int_{b_{\min }( E) }^{b_{\max }( E) }\sqrt{2E -( b-1) ^{2}}db \\
&\sim 4\int_{-\sqrt{2E}}^{\sqrt{2E }}\sqrt{2E -x^{2}}%
dx=8E \int_{-1}^{1}\sqrt{1-x^{2}}dx=4\pi E 
\end{align*}%
since $\int_{-1}^{1}\sqrt{1-x^{2}}dx=\frac{1}{2}\pi$.
Notice that we thus recover the increase of the amplitude  of the oscillations for $a$ and $b$ that follow from 
\eqref{GasymptModified} and \eqref{muone}.


To examine the limit $E\to \infty$, 
 we will use the expression \eqref{changeofenergy} and estimate the different regions of $C(E)$ separately. This ensures, that the energy does not change too much during the cycle
of length 
$T(E)$.

We first consider the region where $b \leq 1$ and $a \in (a_*,E)$ with $1\ll a_* \ll E$. We denote by $b_*$ the value of $b$ such that $a_+(b_*)=a_*$.  Since in this region 
$E \approx a_*-\ln b_*$ this yields $b_* \approx e^{a_*-E}$. Then the contribution to $D(E)$ in this region is 
\[
  2\int_1^{b_*} (b-a) \frac{db}{b}= 2 \int_{b_*}^1 (a-b)\frac{db}{b} \approx 2 \int_{b_*}^1 \frac{a}{b}\,db
  \approx 2 \int_{b_*}^1 \frac{E+\ln b}{b}\,db \approx 2 \int_{b_*}^1 \frac{\ln b}{b}\,db \approx E^2
  \]
   as $E \to \infty$.
We claim that this  gives the main contribution to the change of the energy and are going to show that the remaining regions give smaller contributions 
as $E \to \infty$.

Indeed, consider now the part where $a \in (1,a_*)$. In this region $E \approx a-\ln b$ and if we denote by $\hat b$ the value of $b$ on the cycle that corresponds to $a=1$ we
have $b\approx e^{-E}$. Then  
\[
 2\int_{b_*}^{\hat b}(b-a) \frac{db}{b} \leq C \int_{\hat b}^{b_*} \frac{a}{b}\,db \leq C a_* \ln \hat b \leq a_* E \ll E^2\,.
\]

If both, $a$ and $b$ are small then $E \approx \ln a + \ln b$ and the contribution of the rate of change of energy is estimated by
\[
 C \int_{b_{\min}}^1 \frac{db}{b} \leq C \ln b_{\min} \leq CE \ll E^2\,.
\]

If $b \geq 1$ and if $a$ is small, then the contribution of the change of the energy on the other hand is estimated by
\[
 2\int_1^{b_{\max}}(b-a)\frac{db}{b} \leq C b_{\max} \leq CE\,.
\]
Finally, we consider  the regime where $a,b \geq1$, such that $E \approx a+b$. The corresponding part of $D(E)$ is given by
\[
 \int_{E}^1 (b-a) \frac{2}{b}\,db \approx 2\int_{E}^1 (2b-E)\frac{db}{b} \approx E \ln E \qquad \mbox{ as } E \to \infty\,.
\]

Hence we have derived the approximation
\begin{equation}\label{energychange1}
E_{n+1}-E_{n}\sim \frac{\varepsilon }{4}( E_{n}) ^{2}
\qquad \mbox{ if } \left\vert E_{n}\right\vert \rightarrow \infty\,.
\end{equation}

This approximation is valid as long as $\eps E \ll 1$, that is $a+b \ll \frac{1}{\eps}$ and $a+b \gg e^{-\frac{1}{\eps}}$.

\subsection{Intermediate regime}\label{Ss.intermediate}

The intermediate regime is characterized by the energy $E$ becoming of order $\frac{1}{\eps}$. 
The energy, as defined in \eqref{energydef}, is given in terms of $A$ and $B$ as
$E= \frac{A}{A_{\infty}}+ \frac{B}{B_{\infty}} - \ln (AB) + \ln (A_{\infty}B_{\infty}) -2$. 
Since $A_{\infty}=O(\eps)$ and $B_{\infty}=O(\eps)$ (recall \eqref{Limits}), the fact that the energy is of order $\frac{1}{\eps}$ implies that either
$A=O(1)$, or $B=O(1)$, or $AB \ll1$.

The intermediate regime will be split in four different regions. When $A$ and $B$ are of order one, we can approximate the equation for $G$ by \eqref{A1equation} with $K$ given by
the additive kernel $K(\xi,\eta)=\xi+\eta$.  We call this the {\it kinetic regime}. 
If $A$ is of order one, but $B$ small we can approximate the equation by a simple ODE system for $A$ and $B$, see \eqref{ode1}, that gives an explicit
expression for $G$. If both, $A$ and $B$ are small, we approximate again by the Lotka-Volterra equation \eqref{leadingorderode}, while if $B$ is of order one, but $A$ small, we approximate
by another simple ODE system, see \eqref{ode2}.

\subsubsection{Overview of different regimes}
\label{Ss.overview}

\begin{figure}[h!]
  \centering
  \includegraphics[width=1\textwidth]{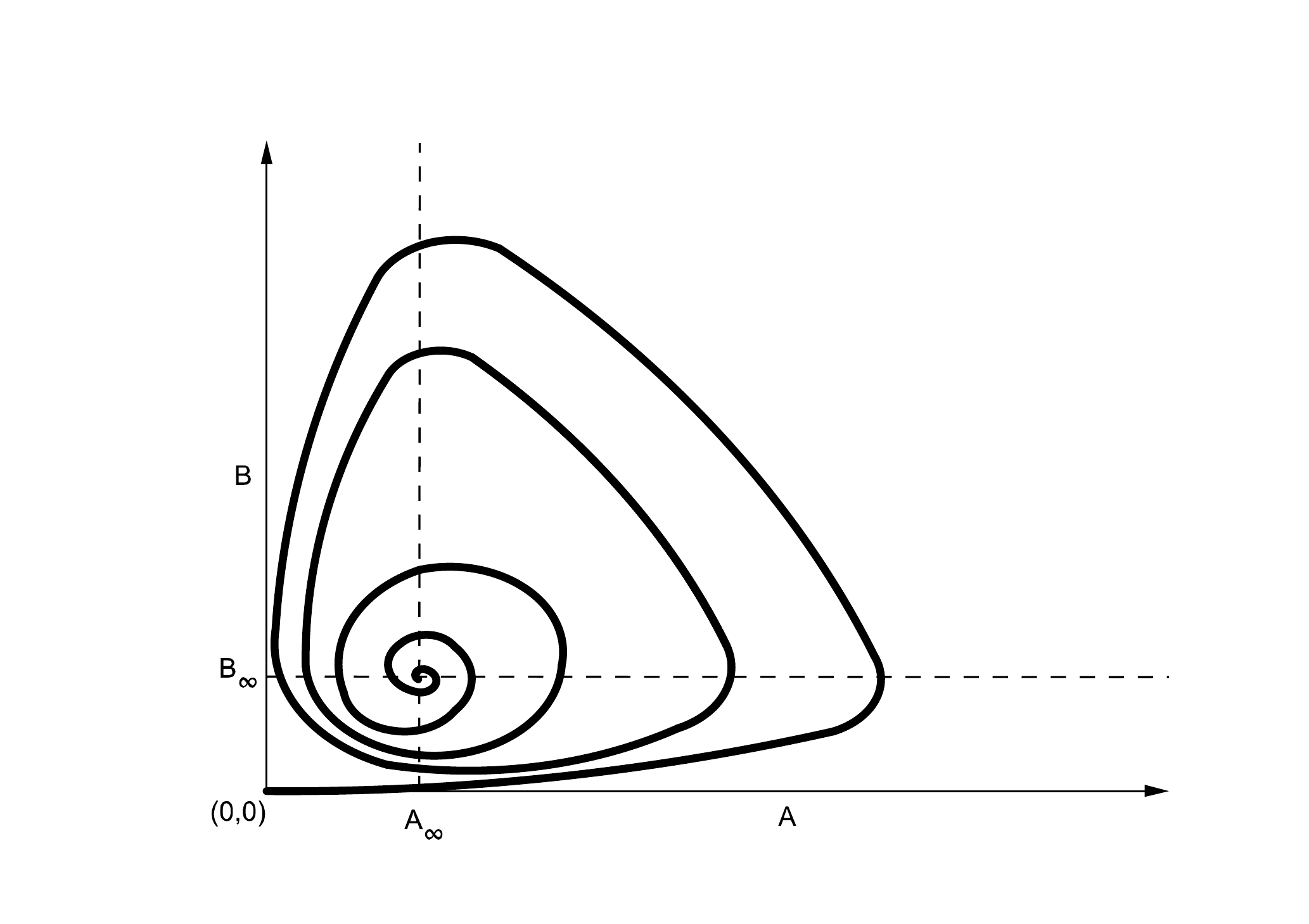}
  \caption{Cartoon of the solution in the $A-B$-plane.}
  \label{fig:phaseplane}
\end{figure}
If we make a plot of the solution in the $A$-$B$-plane (see Figure \ref{fig:phaseplane} for a cartoon), 
it turns out that also in the regime where $E =O(\frac{1}{\eps})$  the  curve ${A(x),B(x)}$ continues to spiral around the point $(A_{\infty},B_{\infty})$ with increasing amplitude in the clockwise
sense for increasing $x$. This curve intersects the line $A=A_{\infty}$ and $B=B_{\infty}$ at consecutive points $x=x_n$ and $x= {\bar x}_n$ where $B(x_n)$ and $A(\bar x_n)$ are 
of order one, respectively. 

In the interval $(x_n,{\bar x}_n)$ the solution can be approximated by a solution to the kinetic equation with additive kernel. In the interval $({\bar x}_n,x_{n+1})$ 
the solution will be approximated
by solutions of certain ODEs. We will see that the length of these intervals is  $|x_n-{\bar x}_n| =O(\ln \frac{1}{\eps})$, while $|{\bar x}_n -x_{n+1}|=O(\frac{1}{\eps^2})$.

\subsubsection{$A,B$ of order one: the kinetic regime}\label{Ss.kinetic}

We now consider equation \eqref{ODE1}-\eqref{Ident} for $\eps=0$,
that is
\begin{align}
\frac{dA(x)  }{dx}   =\frac{1}{2}G(x) \,,& \qquad  \frac{dB(x)  }{dx}%
=  -\frac{1}{2}G(x)
\label{A1}\\
\frac{dP(x)  }{dx}   =-P(x)  +G(x) \,,&\qquad \frac{dQ(x)}{dx}=  Q(x)  -G(x) \label{A2}%
\end{align}
and 
\begin{equation}
G(x)  =4A(x)  B(x)  +P(x)Q(x)  +J[G]  (x) \,,\label{A3}
\end{equation}
with $J[G](x)=\int_{-\infty}^0 dy\int_{\ln(1-e^y)}^{0} dz e^{y-z} G(y+x)G(z+x)$.
\begin{lemma}\label{L.kinetic}
 For any $\rho \in (0,2)$ there exists a continuous solution $(A_{\rho},B_{\rho}, P_{\rho}, Q_{\rho})$ to \eqref{A1}-\eqref{A3} such that the corresponding $G_{\rho}$ 
 satisfies  for $\rho \in (0,2)$, $\rho \not=1$,   the asymptotics 
 \begin{equation}\label{Grhominus}
 G_{\rho}(x) \sim \frac{\rho}{1+\rho} \frac{\sin\big( \frac{\pi \rho}{1+\rho}\big) \Gamma \big( \frac{1}{1+\rho}\big)}{\pi (1+\rho)} e^{\frac{\rho}{1+\rho}x} \qquad \mbox{ as } x \to -\infty
\end{equation}
and
\begin{equation}\label{Grhoplus}
 G_{\rho}(x) \sim \frac{\rho}{1+\rho} \frac{\Gamma(1+\rho)\sin(\pi(1-\rho))}{\pi}  e^{-\rho x} \qquad \mbox{ as } x \to \infty\,.
\end{equation}
For $\rho \in (0,1]$ the solution is positive. Furthermore, for  $\rho=1$, the solution is explicit 
 \begin{equation}\label{Gone}
 G_1(x)=\tfrac{1}{2} \tfrac{1}{\sqrt{2\pi}} e^{\frac{x}{2}} e^{-\frac{e^{x}}{2}}\,.
\end{equation}
 For these solutions we also have that $M_{\rho}:=\int_{-\infty}^{\infty} G_{\rho}(x)\,dx=\frac{\rho}{1+\rho}$. Moreover it holds
 \begin{equation}\label{ABGrelation}
  A_{\rho}(x)= \frac{1}{2}\int_{-\infty}^x G_{\rho}(y)\,dy \, \qquad \mbox{ and } \qquad B_{\rho}(x)= \frac{1}{2}\int_x^{\infty} G_{\rho}(y)\,dy.
 \end{equation}
 We also have the representation formulas
 \begin{equation}\label{B1}
P_{\rho}(x)  =\int_{-\infty}^{x}G_{\rho}(y)  e^{y-x}dy \qquad \mbox{ and }  \qquad  Q_{\rho}(x)  =\int_{x}^{\infty}G_{\rho}(y)  e^{x-y }dy\,.
\end{equation}
\end{lemma}
The above result was proved in \cite{MePe04} for $\rho \in (0,1]$. However, inspection of the formula for the Laplace transform of the solutions, given in \cite{MePe04},
reveals that the result also holds for $\rho \in (1,2)$. 
 It was proved in \cite{MePe04} that the functions $G_{\rho}(x) $ are nonnegative for $\rho \in (0,1]$, while the asymptotics (\ref{Grhoplus}) 
 imply that $G_{\rho }(x) <0$ for large $x$ if $\rho \in ( 1,2) $.
 We also note that $M_1=\frac{1}{2}$ while $M_{\rho} \to \frac{2}{3}$ as $ \rho \to 2$.

\ignore{
One also finds the representation
formula \cite{MePe04}
\begin{equation}\label{Grhorep}
G_{\rho}(X)=\frac{\rho}{1+\rho} \frac{1}{\pi} \sum_{k=1}^{\infty} \frac{(-1)^{k-1}}{k!} e^{k\frac{\rho}{\rho+1}X }\Gamma\Big(1+k-k\frac{\rho}{1+\rho}\Big)\sin\Big(k \pi \frac{\rho}{1+\rho}\big)\,.
\end{equation}
We will later use the asymptotics of $G_{\rho}$ as $\rho \to 0$. In fact, in this case we obtain from \eqref{Grhorep} 
that
\begin{equation}\label{Grhoasymp}
 \begin{split}
G_{\rho}(X) &\approx \frac{\rho}{1+\rho}\sum_{k=1}^{\infty} (-1)^{k-1} e^{k \frac{\rho}{1+\rho }X} k \tfrac{\rho}{1+\rho} 
= \frac{\rho}{1+\rho}\frac{d}{dX} \Big(- \sum_{k=0}^{\infty} (-1)^k e^{k \frac{\rho}{1+\rho}X}\Big)\\
&=\frac{\rho}{1+\rho} \frac{d}{dX} \Big( -\frac{1}{1+ e^{ \frac{\rho}{1+\rho}X}}\Big) = \Big(\frac{\rho}{1+\rho}\Big)^2 \frac{e^{ \frac{\rho}{1+\rho}X}}{\Big(1+e^{ \frac{\rho}{1+\rho}X}\Big)^2}\,.
 \end{split}
\end{equation}
}
We now explain how to use the solutions from Lemma \ref{L.kinetic} to approximate the function $G$  in the intervals $(x_n,{\bar x}_n)$.
Due to \eqref{ABGrelation} we define $M_n=2 B(x_n)$ and approximate the solution in the interval  $(x_n,{\bar x}_n)$ by $G(x)=G_{\rho_n}(x-{\hat x}_n)$, where $M_n=\frac{\rho_n}{1+\rho_n}$.
The point ${\hat x}_n$ is determined by the condition
\begin{equation}\label{match1}
 \frac{1}{2} \frac{\sin\big( \frac{\pi \rho}{1+\rho_n}\big) \Gamma \big( \frac{1}{1+\rho_n}\big)}{\pi (1+\rho_n)} e^{\frac{\rho_n}{1+\rho_n}(x_n-{\hat x}_n)} = A_{\infty} \approx \frac{\eps}{4}\,.
\end{equation}
This condition is obtained by matching the asymptotics of $G_{\rho_n}(x-{\hat x}_n)$, see \eqref{Grhominus}, for $x<{\hat x}_n$, $|x-{\hat x}_n|\gg 1$ 
with $A(x_n)=A_{\infty}$, compare \eqref{ABGrelation}.
Using a similar matching argument in the region $x -{\bar x}_n \gg 1$, we find 
\begin{equation}\label{match2}
 B({\bar x}_n-{\hat x}_n) = \frac{1}{2} \frac{1}{1+\rho_n} \frac{\Gamma(1+\rho_n)\sin(\pi(1-\rho_n))}{\pi}  e^{-\rho_n ({\bar x}_n-{\hat x}_n)} = B_{\infty} = \frac{\eps}{4}\,. 
\end{equation}
Notice that the two matching conditions \eqref{match1} and \eqref{match2} determine ${\bar x}_n$ as a function of $x_n$. We also see that $|x_n-{\bar x}_n| = O(\ln \frac{1}{\eps})$.
Furthermore we  remark that \eqref{ABGrelation} implies that $A({\bar x}_n)=\frac{M_n}{2}$. This is not immediately obvious but follows a posteriori from the fact that the peaks
have distance of order $\frac{1}{\eps^2}$, hence, to leading order $A$ sees only the mass of the last peak.

\subsubsection{An approximation for $G$ if one of $A$ or $B$ is small}
\label{Ss.Gapprox}

We now derive the following key Lemma, that, under certain regularity assumptions,  controls the nonlocal effects by pointwise estimates, which in particular will allow us to 
approximate \eqref{ODE1}-\eqref{Ident} by systems of ODEs in  regions where $G$ is small and where the kinetic approximation  with $\eps=0$  fails.

\begin{lemma}\label{L.Gapprox}
Suppose that $A,B,P$ and $Q$ are solutions to \eqref{ODE1} and \eqref{ODE2}  and $G$ satisfies \eqref{LimitG}.
Furthermore, assume that
for some constants $k_1,k_2$ and $k_3$ with  $k_1 \in (0,1{-}2\delta]$, for some $\delta>0$,  the function $G$ satisfies
 \begin{equation}\label{G1}
 \begin{split}
  \ln (G(x+y)) - \ln (G(x)) &\leq k_1 |y| + k_2 \,, \qquad \mbox{ for all } y \leq -1\,,\\
  \big| \ln (G(x+y)) - \ln (G(x))\big| &\leq k_3 |y|\,, \qquad \mbox{ for all } y \in (-1,0)\,,
  \end{split}
 \end{equation}
then 
\begin{equation}\label{Gapprox}
\big | G(x) -\big(4A(x) B(x) +P(x)Q(x)+2G(x)A(x)\big)\big| \leq C_\delta G(x)^{2}\,.
\end{equation}
If furthermore 
\begin{equation}\label{MLassump}
 |P(x)| \leq 2|G(x)| \qquad \mbox{ and } \qquad |Q(x)| \leq 2|G(x)| 
\end{equation}
then we have
\begin{equation}\label{Gdecr}
\big | G(x) -\big(4A(x) B(x) +2G(x)A(x)\big)\big| \leq C_\delta G(x)^{2}\,.
\end{equation}
\end{lemma}

\begin{proof}
We rewrite \eqref{Ident}  as
\begin{equation*}
G(x) =4A(x) B(x) + P(x)Q(x) +\mathcal{L}_{1}[G](x) +\mathcal{L}_{2}[G] (x)\,,
\end{equation*}%
with
\begin{align*}
\mathcal{L}_{1}[ G] (x) &=\int_{-\infty
}^{0}dy\int_{\ln  1-e^{y}) }^{0}dze^{( 1-\frac{\varepsilon}{2}) (y-z) }G( y+x) G( z+x)\,, \\
\mathcal{L}_{2}[ G] (x) &=\int_{-\infty}^{0}dy\int_{\ln ( 1-e^{y}) }^{0}dze^{\frac{\varepsilon }{2}(y-z) }G( y+x) G(z+x)\,.
\end{align*}
We first show that $|\mathcal{L}_{2}[G]| \leq C_{\delta} G(x)^2$.
Using  assumption \eqref{G1} we obtain  
\begin{align*}
\mathcal{L}_{2}[G](x)  &=
\int_{-\infty}^{0}dz G(x+z)\int_{\ln(1-e^{z}) }^{0}dY G(x+y) e^{\frac{\eps}{2}(y-z)}\\
& \leq C G(x)^2 \Big(\int_{-\infty}^{-1}dz e^{-(k_1+\frac{\eps}{2})z }\int_{\ln (1-e^z)}^0 dy e^{-k_3y}\\
&
 \qquad + \int_{-1}^0 dz e^{-(k_3+\frac{\eps}{2})z}\Big( \int_{\ln(1-e^z)}^{-1} dy e^{-(k_1-\frac{\eps}{2})y} +1\big)\Big)\\
 & \leq C G(x)^2 \Big(\int_{-\infty}^{-1} dz e^{-(1-\delta)z} \int_{\ln ( 1-e^{z}) }^{0}dy e^{-k_3y} \\
 &\qquad 
 + \int_{-1}^0 dz e^{-(k_3+\frac{\eps}{2})z}\Big( \int_{\ln(1-e^z)}^{-1} dy e^{-(1-\delta)y} +1\big)\Big) \leq C_{\delta} G(x)^2
\end{align*}

With $R(x,z):=\int_{\ln ( 1-e^{z}) }^{0}dy  \frac{G(x+y)}{G(x)} e^{(1-\frac{\eps}{2}) y}$ we
rewrite $\mathcal{L}_{1}[G](x)$ as follows
\begin{align*}
\mathcal{L}_{1}[G]&(x) = G(x) \int_{-\infty}^0 dz G(x+z) e^{-(1-\frac{\eps}{2})z}\int_{\ln ( 1-e^{z}) }^{0}dy \frac{G(x+y)}{G(x)}e^{(1-\frac{\eps}{2}) y} \\
&  = G(x) \int_{-\infty}^{-1} dz G(x+z) e^{-(1-\frac{\eps}{2})z}R(x,z) +
 G(x) \int_{-1}^0 dz G(x+z) e^{-(1-\frac{\eps}{2})z}R(x,z)\\
 &=: \mathcal{L}_{1,1}[G](x) + \mathcal{L}_{1,2}[G](x)\,.
\end{align*}
Using \eqref{G1} we obtain $|R(x,z)| \leq Q(x,0) \leq C_{\delta}$ and then also  that  $|\mathcal{L}_{1,2}[G](x)| \leq C_{\delta}G(x)^2$.

To compute $\mathcal{L}_{1,1}[G]$ note first that since $z \leq -1$, we have $y \in (-1,0)$ in the integral for $R$. Hence, due to the second estimate in \eqref{G1} we have
$|R(x,z)-e^z| \leq Ce^{2z}$. Thus  we find, using again \eqref{G1}, that
\[
 \Big| \mathcal{L}_{1,1}[G](x) -G(x) \int_{-\infty}^{-1} G(x+z) e^{\frac{\eps}{2}z}\Big| \leq C_{\delta} G(x)^2
\]
In the previous formula we can replace the integral $\int_{-\infty}^{-1}$ by $\int_{-\infty}^0$, which introduces, due to \eqref{G1}, an error of order $G(x)^2$. 
Since $G$ satisfies \eqref{LimitG}, it turns out that the solution to the first equation in \eqref{ODE1} is given by the first formula in \eqref{Int1}. Thus, the above estimate implies
\eqref{Gapprox}. Estimate \eqref{Gdecr} follows then from \eqref{Gapprox} and the additional assumption \eqref{MLassump}.
\end{proof}

Lemma \ref{L.Gapprox} implies that,  as long as  assumptions \eqref{G1} and \eqref{MLassump} of the Lemma are satisfied, we have
\begin{equation}\label{ODEsapproximation}
\Big |\frac{dA}{dx}+ \frac{\eps}{2}A - \frac{2AB}{1-2A} \Big|  \leq C \frac{G^2}{1-2A}\,,\quad \qquad 
\Big |\frac{dB}{dx} - \frac{\eps}{2}B + \frac{2AB}{1-2A} \Big| \leq C \frac{G^2}{1-2A}\,.
\end{equation}

In particular, we can use the result of Lemma \ref{L.Gapprox} to approximate \eqref{ODE1}-\eqref{Ident} by the  system of ODEs 
\begin{equation}\label{Gdecreasing}
\frac{dA}{dx}=\frac{2AB}{1-2A}-\frac{\varepsilon }{2}A\,, \qquad \quad  \frac{dB}{dx}=\frac{\varepsilon }{2}B-\frac{2AB}{1-2A}
\end{equation}
for any value of $x$ such that $G(x)$ is small.
To make the argument self-consistent, we will have to check afterwards
that the solutions of these approximations
satisfy  the assumptions of Lemma \ref{L.Gapprox}.

\subsubsection{ODE regime 1: $B$ small}\label{Ss.ode1}

We consider now the region where $x={\bar x}_n + O(1)$. From the results of Section \ref{Ss.kinetic} we have that 
$A({\bar x}_n) = \frac{M_n}{2}=\frac{\rho_n}{2(1+\rho_n)}<\frac{1}{4}$ and $B=O(\eps)$.
 Due to \eqref{Grhoplus} the assumption \eqref{G1} of Lemma \ref{L.odeapproximation} holds
 with $k=\rho_n$ and we can use the approximation \eqref{Gdecreasing}.

It turns out that $A$ and $B$ change on different  scales. Notice that $%
A $  increases as long as $\frac{2B}{1-2A}>\frac{\varepsilon }{2}$ and
decreases for $\frac{2B}{1-2A}<\frac{\varepsilon }{2}$. As long as $A$
is of order one we obtain that $B$ decreases exponentially on the time scale
for which $x$ is of order one while  $A$ changes very slowly.
Initially the term $\frac{2AB}{1-2A}$ is relevant. However as soon as $B$
becomes significantly smaller than ${\varepsilon} $ we obtain that $A$ changes via
the equation
$\frac{dA}{dx}=-\frac{\varepsilon }{2}A$.

Without loss of generality we can assume, due to the translation invariance of the equation, that ${\bar x}_n=0$. This assumption is made for notational convenience throughout
this and the following Subsections \ref{Ss.absmall}-\ref{Ss.ode2}. For the same reason we will also drop the index $n$ in $\rho_n$ and $M_n$ respectively.

Hence we can assume  $\frac{2B(0) }{1-2A(0) }=\frac{\varepsilon }{2}$.
Notice that $A$ is very close to $M$ as long as $x \ll \frac{1}{\eps}$. Given that $B$ decreases for $x$ of order one and $A$ changes very little, we obtain the
approximate equations
\begin{equation}\label{ode1}
\frac{dA}{dx} =-\frac{\varepsilon }{2}A \,, \qquad \frac{dB}{dx} =-\frac{2AB}{1-2A}\,,
\end{equation}
whence
\begin{equation}\label{Aapprox}
A(x) =\frac{M}{2}\exp \big( -\frac{\varepsilon }{2}x\big).
\end{equation}
Since 
\[
\int_{0}^{x}\frac{2A (u) du }{1-2A(u) }
=-\frac{2}{\varepsilon }%
\int_{M}^{M\exp ( -\frac{\varepsilon }{2}x) }\frac{dt}{1-t} 
=\frac{2}{\varepsilon }\Big( \ln ( 1-M\exp ( -\frac{\varepsilon }{2}x) ) -\ln ( 1-M) \Big)
\]
we obtain for $B$
\begin{equation}\label{Bapprox}
B(x) =\frac{\varepsilon }{4}\frac{(1-M) (1-M) ^{\frac{2}{\varepsilon }}}{( 1-M\exp( -\frac{\varepsilon }{2}x))^{\frac{2}{\varepsilon }}}\,.
\end{equation}
Since $G$ is small, we can use the approximation $G=\frac{4AB}{1-2A}$. We need to check that $G$ satisfies assumption \eqref{G1} of Lemma \ref{L.odeapproximation}.
With the above computations
\begin{equation}\label{Gapprox1}
  G(x) =\frac{2MB(0)}{1-Me^{- \frac{\eps}{2}x} }\exp\Big( -\frac{\eps}{2}x - \int_0^x \frac{Me^{-\frac{\eps}{2}\xi}}{1-Me^{-\frac{\eps}{2}\xi}} \,d\xi\Big)
\end{equation}
such that
\[
\begin{split}
 \big |\ln G(x+y)-\ln G(x)\big| &\leq C \eps |y|+  \frac{\eps}{2}|y| + \int_{x}^{x+y} \frac{Me^{-\frac{\eps}{2}\xi}}{1-Me^{-\frac{\eps}{2}\xi}} \,d\xi\\
 & \leq C \eps |y| + \frac{M}{1-M}|y| \leq k_1|y|\,,
\end{split}
\]
where we can choose $k_1<1$ since due to $M< \frac{1}{2}$ we have $\frac{M}{1-M}< 1$. Hence assumption \eqref{G1} of Lemma \ref{L.Gapprox} is satisfied at least for $y>-x$.
If $y\leq -x$ we can use the asymptotics of the solutions in the kinetic regime, see \eqref{Grhoplus}, which gives \eqref{G1} since $\rho<1$.


\subsubsection{$A,B$ small: matching ODE 1 regime with Lotka-Volterra }\label{Ss.absmall}

We now consider the matching between the solutions of the previous subsection and the regime where $A$ and $B$ are small, and hence the solution
 behaves as a solution to the  Lotka-Volterra system \eqref{leadingorderode}. 
For this purpose recall the relation between $A,B$ and $a,b$, see \eqref{rescaling1}.

We enter  the Lotka-Volterra regime if $a$ becomes of order
one, i.e. $A$ of order $\varepsilon$. In this range $b$ is very small. 
We define  $\tilde x$ 
 to be the time when $a$ becomes one, i.e. by
$\frac{M}{2}\exp ( -\frac{\varepsilon }{2}\tilde{x}) =\frac{2G_{\infty }}{\varepsilon }\sim \frac{\varepsilon }{2}$
and then introduce the new variable $t=\varepsilon\big(x-\tilde x\big)$. Using \eqref{Aapprox} we then find
\begin{equation*}
a(t ) =\frac{\varepsilon }{2G_{\infty }}A(x) \sim 
\frac{\varepsilon M}{4G_{\infty }}\exp ( -\frac{\varepsilon }{2}x) =\frac{\varepsilon M}{4G_{\infty }}
\exp ( -\frac{\varepsilon }{2}\tilde{x}) \exp  (-\frac{t }{2}) =\exp ( -\frac{t }{2})
\end{equation*}%
as $t \rightarrow -\infty$, which  gives the matching condition
\begin{equation}\label{amatching}
a(t) \sim \exp ( -\frac{t }{2}) \qquad \mbox{ as } t \rightarrow -\infty\,.
\end{equation}
In order to compute the asymptotics of $b(t)$ we use
\[
b(t ) =\frac{\varepsilon }{2G_{\infty }}B(x) =%
\frac{\varepsilon ^{2}}{8G_{\infty }}\frac{(1-M) ^{\frac{2}{%
\varepsilon }+1}}{( 1-M\exp ( -\frac{\varepsilon }{2}\tilde{x}) \exp ( -\frac{t }{2}))^{\frac{2}{\varepsilon }}%
} \approx \frac{1}{2}\frac{(1-M)^{\frac{2}{\varepsilon }+1}}{( 1-\frac{2G_{\infty }}{\varepsilon }\exp ( -\frac{t }{2}))^{\frac{2}{\varepsilon }}}\,.
\]
In the limit $\varepsilon \rightarrow 0$ we obtain the approximation for large $t <0$ (but of order one)
\begin{align*}
\frac{1}{( 1-\frac{2G_{\infty }}{\varepsilon }\exp ( -\frac{t }{2}))^{\frac{2}{\varepsilon }}} &=\exp \Big( -\frac{2}{\varepsilon }
\ln \big( 1-\frac{2G_{\infty }}{\varepsilon }\exp\big ( -\frac{t }{2}\big)\big)\Big) \\
&\approx \exp \Big( \frac{2}{\varepsilon }\frac{2G_{\infty }}{\varepsilon }%
\exp \big( -\frac{t }{2}\big) \Big) 
\approx \exp \Big( \exp \big( -\frac{t }{2}\big) \Big)\,,
\end{align*}
which implies the matching condition
\begin{equation}\label{bmatching}
b(t ) \sim \frac{1}{2}(1-M)^{\frac{2}{\varepsilon }+1}\exp
\Big( \exp \big( -\frac{t }{2}\big) \Big) \qquad \mbox{ as } 
t \rightarrow -\infty\,.
\end{equation}
We recall that due to \eqref{ODEsapproximation} and \eqref{rescaling1}  the equations in this region are 
approximated to leading order by \eqref{leadingorderode} which has  the conserved energy $E=a+b-\ln(ab)-2$.
We obtain that in the matching region, given the
smallness of $b$ and the fact that $a$ is large $E\sim a-\ln b$,
whence
\begin{equation}\label{energymatch}
E\sim \exp \big( -\frac{t }{2}\big) - \frac{2}{\varepsilon } \ln (1-M) -\exp \big( -\frac{t }{2}\big)
=- \frac{2}{\varepsilon } \ln (1-M)\,.
\end{equation}

\subsubsection{Lotka-Volterra; transition time}\label{Ss.transition}

In the region where $a+b=O(1)$ we approximate the dynamics of \eqref{ODE1}-\eqref{Ident} by the Lotka-Volterra system \eqref{leadingorderode}
together with the matching conditions \eqref{amatching} and \eqref{bmatching}.

We need to estimate the time that the trajectory spends in the region where $a,b=O(1)$. Due to the invariance of 
 the Lotka-Volterra equation  under the
transformation
$( a,b,t ) \rightarrow ( b,a,-t )$ it suffices to compute the time
that the trajectory needs to arrive to the line $\{ a=b\}$.
The key simplification is that in all the time required to bring the
trajectory from the asymptotics \eqref{amatching}, \eqref{bmatching}  to the line $\{ a=b\} $
we have $b\ll 1.$ Then, we have the approximations
$\frac{da}{dt }=-\frac{a}{2}$ and $\frac{db}{dt }=\frac{b(1-a) }{2}$
such that 
$a(t) =\exp \big( -\frac{t }{2}\big)$.

On the other hand the Lotka-Volterra equations and \eqref{energymatch} imply
\begin{equation}\label{energy1}
a-\ln ( a) -\ln ( b) =- \frac{2}{\varepsilon } \ln ( 1-M)
\end{equation}
and if  $a=b$ this implies
$a-2\ln  a =-\frac{2}{\varepsilon } \ln (1-M)$.
As a consequence  we obtain for $a$ at the line $\{a=b\} $ the following asymptotics
$a=( 1-M) ^{ \frac{1}{\varepsilon } }$.

This implies that  $a$ reaches an extremely small value. 
 Using $a(t) =\exp \big( -\frac{t }{2}\big) $
we obtain for the time $\hat t$ to arrive to the line $\{ a=b\} $ can be
approximated as 
\begin{equation}\label{tauhat}
\hat{t} = \frac{2}{\varepsilon } \ln \Big( \frac{1}{1-M}\Big).
\end{equation}

Due to \eqref{ODEsapproximation}, the original equation \eqref{ODE1}-\eqref{Ident} can be approximated by the Lotka-Volterra
equation with an error on the right hand side of the order $\eps ab$.  Therefore it follows that $\frac{dE}{dt} = O\big(\eps(a+b)\big)$.
Since $b<a$, and $a(t) =e^{-\frac{t}{2}}$, $b=O(e^{-\frac{C}{\eps}})$ it follows from \eqref{energy1} that the change of the energy until $\hat t$ is of order $\eps$ and since the energy is 
of order $\frac{1}{\eps}$ we can assume that the energy is approximately
constant.

We can now compute the asymptotics of the solution when $t - \hat t \gg 1$ in order to obtain the matching condition with the next region.
Neglecting  also the term
$\ln b $ compared to $b$ we obtain, if $b \gg 1$, the approximation 
$a\sim e^{-E}e^{b}$
and using \eqref{energy1}  it
follows that
\begin{equation}
a\sim ( 1-M) ^{ \frac{2}{\varepsilon } }e^{b}\,, \qquad b \gg 1\,.  \label{B5}
\end{equation}
In this range of values of $a,b$ we can use the
approximation $\frac{db}{dt }=\frac{b}{2}$, hence $b(t)=Ce^{\frac{t}{2}}$. Since $b(\hat t)=a(\hat t)=e^{-\frac{\hat t}{2}}$ we have
\begin{equation}\label{B6}
 b(t)=e^{\frac{t-2\hat t}{2}}\,.
\end{equation}
Notice, that the assumption \eqref{G1} of Lemma \ref{L.Gapprox} is satisfied, since the function $\ln G(\cdot)$ is a function of $t$ and hence the derivative of this function
with respect to $x$ is of order $\eps$.

\subsubsection{ODE regime 2: $A$ small}\label{Ss.ode2}

We now describe the region where $A$ is small and $B$ increases up to values of order one.

We recall \eqref{ODEsapproximation}  and given that in this region  $A \ll \eps $ is small and  $B \gg \eps$
we obtain
\begin{equation}\label{ode2}
\frac{dA}{dx}=AB\,,\qquad  \quad \frac{dB}{dx}=\frac{\varepsilon }{2}B
\end{equation}
which implies 
$B=C_{0}e^{\frac{\varepsilon x}{2}}$,
where $C_{0}$ has to be determined by matching with \eqref{B5} and \eqref{B6} . This gives $\frac{dA}{dx}=C_{0}e^{\frac{\varepsilon x}{2}}A$
and hence
$A(x) =C_{1}\exp \big( \frac{2C_{0}}{\varepsilon }e^{\frac{\varepsilon x}{2}}\big)$.
To determine $C_0$ and $C_1$ we recall \eqref{rescaling1}, $G_{\infty} \approx \frac{\varepsilon^2}{4}$, the definition $x= \tilde x + \frac{t}{\eps}$ and 
denote $x^*=\tilde x + \frac{2\hat t}{\varepsilon}$. We obtain
 up to exponential accuracy that $B( x) \sim \frac{\varepsilon }{2}e^{\frac{\varepsilon (x-x^{\ast }) }{2}}$ and
\begin{equation}\label{ABode2}
A(x) \sim \frac{\varepsilon }{2}( 1-M)^{ \frac{2}{\varepsilon } }\exp \Big( e^{\frac{\varepsilon
( x-x^{\ast }) }{2}}\Big)= \frac{\varepsilon }{2}( 1-M)^{ \frac{2}{\varepsilon } }\exp \Big ( \frac{2}{\eps} B(x)\Big) \,.
\end{equation}
Due to \eqref{Gdecreasing} 
 the asymptotics above is valid as long as $A \ll \frac{\varepsilon }{2}$. We denote by $x_{n+1}$ the point when $A(x_{n+1})=A_{\infty} \approx \frac{\varepsilon}{2}$.
 Notice that $B(x_{n+1})$ is maximal and decreases afterwards. Using \eqref{ABode2} simple rearrangements
 give that in the limit $\varepsilon \to 0$ we obtain
\begin{equation}\label{massincrease}
B(x_{n+1}) = -\ln ( 1-M).
\end{equation}
This formula yields  the desired iterative condition for the masses.
Starting with the mass $M$ we reach after one cycle the new value of the mass
$\ln ( \frac{1}{1-M}) >M$.

Notice also that the first inequality in \eqref{G1} is satisfied since $G$ is increasing in the region that we consider in this section. The second inequality follows from 
\[
 \big| \ln G(x+y) - \ln G(x) \big| \leq \frac{2}{\eps} \Big | B(x) \big(e^{\frac{\eps}{2}y}-1\big) \Big| + \frac{\eps}{2}|y| \leq \Big(B(x_{n+1})+ \frac{\eps}{2}\Big)|y|\,.
\]

The analysis of this subsection is similar to the one in Subsection \ref{Ss.ode1}. However, there $A$ was of order one, while here it is small. As a consequence the value $B(x_{n+1})$ of $B$
at the end of this region is different from 
$A(0)=A(x_n)$ in Subsection \ref{Ss.ode1}. Due to this fact, the amplitude of the oscillations in the $A$-$B$-plane is increasing.

\subsubsection{Summary of the intermediate regime}

We have identified successive points $x_n,\bar x_n$ and  $x_{n+1}$ such that
$A(x_n)=A(x_{n+1})=A_{\infty} \approx \frac{\eps}{2}$ and $B(\bar x_n)=B_{\infty} \approx \frac{\eps}{2}$. We found that $\bar x_n-x_n= O(\ln \frac{1}{\eps})$ and $x_{n+1}-\bar x_n=
O(\frac{1}{\eps^2})$. In $[x_n,\bar x_n]$ the solution $G$ is of order one and is to leading order given  by a solution to the kinetic equation \eqref{A1}-\eqref{A3} given in 
Lemma \ref{L.kinetic} with $\rho_n$ such that $2B(x_n)=M_n = \frac{\rho_n}{1+\rho_n}$. In $[\bar x_n,x_{n+1}]$, where $G$ is small, the solution can be approximated by three different
simple ODE systems. The key finding in the previous subsections is a formula
for the increase of the amplitude of the oscillations in the $A$-$B$-plane  characterized  by the numbers $M_n$  for which, due to \eqref{massincrease}, we have the recursive formula 
\begin{equation}\label{Miteration}
M_{n+1}= \ln \frac{1}{1-M_n}\,.
\end{equation}
We emphasize, that the value of $M$ does not change during the kinetic regime nor the Lotka-Volterra regime, but that the change is due to the asymmetry in the equations
that describe  ODE regime 1 and  ODE regime 2 respectively. 

In the limit $M_n\to 0$, formula \eqref{Miteration} implies $M_{n+1}-M_n \sim \frac{1}{2}M_n^2$. On the other hand, we have in this regime 
that $E \sim \frac{4}{\eps}(A+B) \sim \frac{4B}{\eps}\sim \frac{2}{\eps}M$. Hence, \eqref{Miteration} agrees with \eqref{energychange1} in the limit $M_n\to 0$. 

\subsection{Shooting argument}\label{Ss.shooting}

We will argue in Section \ref{S.fast}  that the constant $C_{\ast}$ in \eqref{GasymptModified}
can be chosen such that no fast instabilities develop for any finite
$x.$ 
We will show now that under the assumption that no such `dormant' instability occurs, 
the constant $L$ can be selected in the
interval $[1,e^{\frac{\pi\varepsilon}{2}})  $ in such a way that
the resulting function $G(x)  $ is globally positive and satisfies $G(x)\to 0$ as $x \to \infty$.
More precisely we obtain
\begin{lemma}
 \label{L.Glargex}
 There exists a value $L_{\ast}\in [1,e^{\frac{\pi\varepsilon}{2}})  $ and $\hat x_n\in \R$  such that the corresponding $G$ satisfies 
 \begin{equation}
G(x)  \sim \Big( \frac{1}{2\sqrt{2}}e^{\frac{x-\hat x_n}{2}} + \frac{\eps}{4} e^{x-\hat x_n}\Big)
  e^{-\frac{e^{(x-\hat x_n)}}{2}}\label{AppQFinal}
\end{equation}
for $x\gg \hat x_n$.
\end{lemma}

\begin{proof}
Due to (\ref{Miteration}) it follows
that $M_{n}$ becomes larger than $\frac{1}{2}$ for sufficiently large $n$.
By making a translation of the variable $x$ in (\ref{GasymptModified}) if needed, we can
assume that for some values of $L\in[  1,e^{\frac{\pi\varepsilon}{2}})  $ the corresponding value of $M_{n}$ is larger than $\frac{1}{2}$
and for other values $L\in [1,e^{\frac{\pi\varepsilon}{2}})  $ we
have $M_{n}<\frac{1}{2}.$

We recall that if $M_{n}<\frac{1}{2}$ we can approximate $G(x)$ arguing as in Subsections \ref{Ss.ode1}-\ref{Ss.ode2}. Then $A(x),B(x)$ 
converge to values where $A(x)  =A_{\infty}$ and $B(x)  $ is of order one for
$x=x_{n+1}$. 
On the other hand if $M_{n}>\frac{1}{2}$ 
Lemma \ref{L.kinetic} implies that $B(x)  $ changes sign at some finite $x$. Given that the solutions of (\ref{ODE1}%
)-(\ref{Ident}) depend continuously on the parameter $L$ in
(\ref{GasymptModified}) we  have
the existence of $L=L_{\ast}\in[1,e^{\frac{\pi\varepsilon}{2}})$ for which the corresponding solution of (\ref{ODE1})-(\ref{Ident})
neither changes sign nor arrives to values of $A(x)  =A_{\infty}$
with $B(x)  $ of order one. We will argue that for $L_{\ast}$ we  have $G(x) \to 0$ 
as well as  $(A(x),B(x)) \to (0,0)$ as 
 $x\rightarrow\infty$.

To see this we first remark that due to the arguments above, $M_{n}$ must take
the value $\frac{1}{2}$ or is close to it for small $\varepsilon.$ Then
during the kinetic regime $G(x)  $ is approximated by
$G_{1}(x)  $ (cf. Lemma \ref{L.kinetic}) and at the end of that
phase the asymptotics of $G(x)$ is given by the
right-hand side of (\ref{Gone}). Then $G(x)  $ cannot be  approximated by ODEs as in the case for
$M_{n}<\frac{1}{2},$ but  we will approximate $G(x)  $ by a suitable integral equations. For notational simplicity we assume for the moment that $\hat x_n=0$ and  introduce the
 variables $\xi=e^{x}$, $\eta=e^{y}$,$\zeta=e^{z}$, $\xi^2h(\xi)=G(x)$, 
$\bar{a}(\xi)  =A(x)$, $\bar b(\xi)=B(x)$, $\bar p(\xi)=P(x)$ and $\bar q(\xi)=Q(x)$.
Then (\ref{ODE1})-(\ref{Ident}) becomes
\begin{align}
\xi\frac{d\bar{a}}{d\xi}   =-\frac{\varepsilon}{2}\bar{a}+\frac{1}{2}%
\xi^{2}h \,,& \qquad \qquad  \xi\frac{d\bar{b}}{d\xi}=\frac{\varepsilon}{2}\bar{b}-\frac{1}{2}\xi^{2}h\,, \label{EqA1}\\
\xi\frac{d\bar{p}}{d\xi}   =-(  1-\frac{\varepsilon}{2}) \bar{p}+\xi^{2}h\,,& \qquad \qquad  \xi\frac{d\bar{q}}{d\xi}=\big(  1-\frac{\varepsilon}%
{2}\big)  \bar{q}-\xi^{2}h\,,\label{EqA2}%
\end{align}%
\begin{equation}
\xi^{2}h=4\bar{a}\bar{b}+\bar{p}\bar{q}+J[h]\,,  \label{hForm}%
\end{equation}
where
\begin{equation}
J[h](\xi) =\int_{0}^{\xi}\eta h(\eta) d\eta\int_{\xi-\eta}^{\xi}h(\zeta)(\eta \zeta)^{\frac{\varepsilon}{2}%
}\big(    \eta^{1-\varepsilon}+\zeta^{1-\varepsilon
}\big)  d\zeta\label{IntOp}%
\end{equation}

On the other hand, since $G(x)  $ behaves like the right-hand
side of (\ref{Gone}) we obtain
\begin{equation}
h(\xi)  \sim\frac{1}{2\sqrt{2\pi}}\frac{e^{-\frac{\xi}{2}}}%
{\xi^{\frac{3}{2}}} \qquad \mbox{ for } \xi^{\eps} \sim 1\,, \mbox{ i.e. } \ln \xi \sim \frac{1}{\varepsilon}\,.\label{AsExp}%
\end{equation}
We remark that in the set of values $\xi$ for which the approximation
(\ref{AsExp}) holds we have that $\bar{a}$ is of order one, $\bar{p}$ is of
order $\frac{1}{\xi}$ and $\bar{b},\ \bar{q}$ are of order $\frac{e^{-\frac
{\xi}{2}}}{\sqrt{\xi}}.$ 
Therefore, the terms $4\bar{a}\bar{b}$ and $\bar{p}\bar{q}\ $in
(\ref{hForm}) can be neglected
and (\ref{hForm}) can be approximated by
\begin{equation}
\xi^{2}h=J[h] \,. \label{VoltApp}%
\end{equation}

We now use the approximation (\ref{VoltApp}) to describe how the asymptotics
(\ref{AsExp}) is modified  if
$\xi^{\varepsilon}$ becomes of order one or larger. To this end 
we will look for solutions of (\ref{VoltApp}) of  the form $h(\xi)=r(\xi)e^{-\frac{\xi}{2}}$,
where we assume that $r(\xi)  $ contains only functions
that change algebraically for large $\xi$, i.e. $r(\xi+1) \simeq r(\xi)  $.  Then (\ref{VoltApp})
becomes
\begin{equation}\label{VoltApp2}
 \begin{split}
\xi^{2}r(\xi) & =\int_{0}^{\xi}r(\zeta)d\zeta \int_{\xi-\zeta}^{\xi}\eta r(\eta)(\eta \zeta)^{\frac{\varepsilon}{2}}(\eta^{1-\varepsilon
}+\zeta^{1-\varepsilon}) e^{\frac{1}{2}(\xi-\eta-\zeta)  }d\eta\\
& =\int_{0}^{\xi}r(\zeta)d\zeta \int_{0}^{\zeta}(\xi{-}\zeta{+}\theta)  Q(\xi{-}\zeta{+}\theta)  (   (\xi-\zeta+\theta)\zeta)^{\frac{\varepsilon}{2}}
\big( (\xi-\zeta+\theta)^{1-\varepsilon}+  \zeta^{1-\varepsilon}\big)  e^{-\frac{1}{2}%
\theta}d\theta\,.
\end{split}
\end{equation}

For large values of $\xi$ the function $r(\xi-\zeta+\theta)  $ changes more slowly than 
$e^{-\frac{1}{2}\theta}$ such that 
\begin{equation}\label{VoltApp3}
 \begin{split}
\xi^{2}r(\xi) & =\int_{0}^{\xi}r(\zeta)d\zeta\int_{0}^{\zeta}(\xi-\zeta)r(\xi-\zeta)   (  (\xi-\zeta)\zeta)  ^{\frac{\varepsilon}{2}}
\big( (\xi-\zeta)^{1-\varepsilon}+ \zeta^{1-\varepsilon}\big)  e^{-\frac{1}{2}\theta}d\theta\\
& =2\int_{0}^{\xi}r(\zeta)(\xi-\zeta)r(\xi-\zeta)(  (\xi-\zeta)\zeta)^{\frac{\varepsilon}{2}}\big(  (\xi-\zeta)^{1-\varepsilon}+\zeta^{1-\varepsilon}\big)
\big(  1-e^{-\frac{1}{2}\zeta}\big)  d\zeta\,.
\end{split}
\end{equation}
For large values
of $\xi$ and small values of $\xi$ the main contribution to the integral in
(\ref{VoltApp3}) is due to the values of $\zeta$ satisfying $\zeta\leq\delta \xi,$ for
$\delta$ small. This will be  checked later "a posteriori".
Thus, if $\zeta\leq\delta \xi$ we
have the approximations $(\xi-\zeta)^{1+\frac{\varepsilon}{2}}%
\sim   \xi^{1+\frac{\varepsilon}{2}},\ (\xi-\zeta)^{1-\varepsilon}+ \zeta^{1-\varepsilon}\sim \xi^{1-\varepsilon}$
such that
\begin{equation}
r(\xi) =2\xi^{-\frac{\varepsilon}{2}}\int
_{0}^{\delta \xi}r(\zeta)r(\xi-\zeta)\big( 1-e^{-\frac{1}{2}\zeta}\big)  \zeta^{\frac{\varepsilon}{2}}d\zeta\,.\label{VoltApp4}%
\end{equation}

We recall that (\ref{AsExp}) yields a valid approximate solution of
(\ref{VoltApp4}) if $\xi^{\varepsilon}\sim1.$ For these  $\xi$ we have 
\begin{equation}
r(\xi) \sim\frac{1}{2\sqrt{2\pi}}\frac{1}{\xi^{\frac{3}{2}}}\,. \label{AppQ1}%
\end{equation}

Indeed, plugging (\ref{AppQ1}) into (\ref{VoltApp4}) and using $r(\xi-\zeta) \approx r(\xi)  $ for large $\xi$ we readily obtain that
(\ref{VoltApp4}) reduces to checking the identity
\[
1\approx\frac{2  \xi^{-\frac{\varepsilon}{2}}}{2\sqrt{2\pi}}%
\int_{0}^{\delta \xi}\big(  1-e^{-\frac{1}{2}\zeta}\big)   \zeta^{\frac{\varepsilon}{2}-\frac{3}{2}}d\zeta
\approx\frac{ \xi^{-\frac{\varepsilon}{2}}}{\sqrt{2\pi}}\int_{0}^{\infty}\big(  1-e^{-\frac{1}{2}\zeta}\big)   \zeta^{\frac{\varepsilon}{2}-\frac{3}{2}}d\zeta\,.
\]
Using that $\int_{0}^{\infty}(1-e^{-\frac{1}{2}\zeta})\zeta^{-\frac{3}{2}}d\zeta=\sqrt{2\pi}$ this approximation follows due to the
assumption $\xi^{\varepsilon}\sim1.$
For $\xi^{\varepsilon}\gg1$ a scaling argument
suggests the asymptotics
$ r(\xi) \sim\frac{\bar{C}}{\xi}$
for some $\bar{C}>0$. Plugging this into
(\ref{VoltApp4}) and using also that $r(\xi-\zeta)  \approx r(\xi)  $ we reduce (\ref{VoltApp4}) to
$1\approx2\bar{C} \xi^{-\frac{\varepsilon}{2}}\int_{0}^{\delta
\xi}\big(  1-e^{-\frac{1}{2}\zeta}\big)\zeta^{\frac{\varepsilon
}{2}-1}d\zeta\approx\frac{4\bar{C}}{\varepsilon}$,
such that $\bar {C}=\frac{\varepsilon}{4}.$ 
For $\xi^{\varepsilon}\gg1$ we have thus obtained
\begin{equation}
r(\xi) \sim\frac{\varepsilon}{4\xi}\,.\label{AppQ3}%
\end{equation}
Notice that  the contributions due to the part of the integrals
with $\zeta>\delta \xi$ which have been neglected in (\ref{VoltApp3}) in order to
derive (\ref{VoltApp4}) give a small contribution if $\xi$ is large and
$\varepsilon$ small for both solutions (\ref{AppQ1}) and (\ref{AppQ3}).

We can now combine the solutions (\ref{AppQ1}) and (\ref{AppQ3}) to obtain a
global solution of (\ref{VoltApp4}) that is  valid for arbitrarily large values of $\xi.$
Indeed, it is easy to check that
\begin{equation}
r(\xi) \sim\frac{1}{2\sqrt{2\pi}}\frac{1}{\xi^{\frac{3}{2}}}%
+\frac{\varepsilon}{4\xi}\label{AppQ3a}%
\end{equation}
satisfies \eqref{VoltApp4} approximately for arbitrary values of $\xi$ and $\varepsilon$.

In order to check the consistency of the approximation (\ref{AppQ3a}) we
must check that the terms $4\bar{a}\bar{b}$ and $\bar{p}\bar{q}$ in
(\ref{hForm}) are negligible. It follows from integrating \eqref{EqA1}-\eqref{EqA2},
 using the definition of $r$ and (\ref{AppQ3a}),
 that
$\bar{a}(\xi)  \leq\frac{C}{\xi^{\frac{\varepsilon}{2}}}$, $\bar{b}(\xi)  \leq C\xi h(\xi)$,
$\bar {p}(\xi)  \leq\frac{C}{\xi^{1-\frac{\varepsilon}{2}}}$ and  $\bar{q}(\xi)  \leq C\xi h(\xi)$.
Therefore $4\bar{a}\bar{b}+\bar{p}\bar{q}\ll \xi^{2}h$ for large $\xi$ and the
self-consistency of the argument follows.
It is important to remark, however, that although the term $4\bar{a}\bar{b}$
can be neglected in (\ref{hForm}), the equation for $\bar{b}(\xi)
$ in (\ref{EqA1}) shows that for some solutions of (\ref{ODE1}%
)-(\ref{Ident}) the function $B(x)  $ can start an exponential
growth taking place in lengths of $x$ of order $\frac{1}{\varepsilon}$.
In other words, the solutions of (\ref{ODE1})-(\ref{Ident}) might
contain dormant instabilities 
that result either in a change of sign of $B$ or in $B$
becoming large enough and thus the solution entering another cycle in the $(A,B)$-plane.
Thus, this instability corroborates the fact
that the desired solution  is obtained by a shooting argument and appears 
as the transition between the ODE behaviour described in Subsections \ref{Ss.ode1}-\ref{Ss.ode2} and
$B$ changing sign.

Finally, amending for the fact that we had assumed that $\hat x_n=0$, we obtain from \eqref{AppQ3a} the result  \eqref{AppQFinal}.
\end{proof}

\section{Fast dynamics}\label{S.fast}

We have seen in the asymptotics of the function $G$ in \eqref{GasymptModified} that the term $C_*e^{\mu_*x}$ changes much faster
than the other terms in the formula \eqref{GasymptModified}  due to the  exponentially growing mode for $Q$ in \eqref{ODE2}.
If $C_*$ is of order one, the instability will appear for $x$ of order one. However, if $C_*$ is very small, the exponentially growing
instability  will  only become visible for large values of $x$. 

For the solution computed in  Section \ref{S.wave} the following holds. In all the regions where the solution can be approximated by ODEs 
the function $G(x)$ changes in length scales much larger
than one and in particular, in all those regions we have $|Q-G| \ll Q$.
However, if $C_*$ is sufficiently small, there exists an $x_*\gg 1$
\begin{equation}
Q(x)  =G(x_*)\Big( 1 \pm e^{x-x_* }\Big) +...
\label{Masympt}%
\end{equation}
for $x<x_*$ and  $|x-x_*|\gg 1 $.
 In this case the instability is triggered in $x =x_*$.
We can obtain both, the positive and the  negative sign in \eqref{Masympt} by choosing $|C_*|$ in \eqref{GasymptModified} sufficiently large
and positive or negative respectively.
Moreover, modifying the value of $C_*$ we can tune the value of $x_*$.
It turns out that if the instability is triggered in any region where the solution of \eqref{ODE1}-\eqref{Ident} can be approximated by an ODE, then either $Q$, 
if the sign in \eqref{Masympt}
is minus, or $B$  if otherwise, become negative and hence the solution cannot represent an admissible solution to the coagulation equation (cf. Sections \ref{Ss.fastinstability} 
and \ref{Ss.odeinstability}).

It is also  possible to choose the value of $C_*$ such that the instability is triggered in the kinetic regime, described in Section \ref{Ss.kinetic}. 
In this regime $|Q-G|$ is not small compared to $G$ anymore. Here we have to take into
account that the Volterra problem \eqref{ODE1}-\eqref{Ident} with 
$\eps=0$ has more solutions than those given by Lemma \ref{L.kinetic}. The general solution has the form ${ Q}^{\alpha}_{\rho}(x)$
with ${Q}^{\alpha}_{\rho}(x)-Q_{\rho}(x) \sim \alpha e^x$ as $x \to -\infty$
with arbitrary $\alpha$
and we denote by ${G}^{\alpha}_{\rho}$ the corresponding function $G$. 
Recall that for our special solutions we have $G(x) \sim G_{\rho}(x{-}{\hat x}_n)$ for $|x{-}{\bar x}_n|$ of order one. 
However, the difference $G_{\rho}-{G}^{\alpha}_{\rho}$ differs by an exponentially small amount for $x<\hat x_n$ and  $|x{-}\hat x_n |\gg 1$, hence there is a priori no 
reason to choose $G_{\rho}$ instead of ${G}^{\alpha}_{\rho}$. We will see later, however, in Section \ref{Ss.instkinetic} that if $\alpha \not= 0$,
then either $Q$ or $B$ become negative and again we do not obtain an admissible solution. 

Therefore, adjusting the value of $C_*$ we can obtain a solution for which no instability is triggered neither in the ODE regime, nor in the kinetic regime.

\subsection{Fast instabilities appearing during the Lotka-Volterra regime.}\label{Ss.fastinstability}

We first assume that the instability is triggered in the Lotka-Volterra regime, see Section \ref{Ss.lotkavolterra}, that is the regime with
energy $E \ll \frac{1}{\eps}$.

 Suppose that the instability appears at a given $x_*$, for which the ODE approximation is still valid. This would mean that the
asymptotics (\ref{Masympt}) holds for $x<x_*$ and $\frac{1}{\eps} \gg | x-x_* | \gg 1$.

If \eqref{Masympt} holds with a minus sign, we can approximate  \eqref{ODE1}-\eqref{Ident} as follows. 
Lemma \ref{L.Gapprox} implies that $J[G] \leq C (G^2+GA)$. In the Lotka Volterra regime we have $G\ll 1$ and $A \ll 1$ and hence
$J[G]$ can be neglected in \eqref{Ident}.
As long as $G$ remains approximately constant, we can approximate $Q$ by
\begin{equation}\label{Nminus}
 \frac{dQ}{dx} = \big(1-\frac{\eps}{2}\big) Q -G\,,\qquad Q(x) = G(x_*)\Big( 1-e^{(1-\frac{\eps}{2})(x-x_*)}\Big)\,,
\end{equation}
hence $Q$ becomes negative for $x=x_*$. Since $Q$ is decreasing, the term $PQ$ in \eqref{Ident} remains negligible, the other variables essentially do not 
change and the approximation is self-consistent.

If \eqref{Masympt} holds with a plus sign, the function $Q$ becomes of order one if $x \sim x_1$ where 
\begin{equation}\label{X1}
G(x_*) e^{(x_1-x_*)}=1\,.
\end{equation}
Since $G(x_*)$ can be very small if the energy $E$ is large, it might happen that $|x_1-x_*| \to \infty$ as $\eps \to 0$. During this part of the evolution 
the variables $a$ and $b$ are described by \eqref{leadingorderode} and hence if $|x_1-x_*|$ is large, we might have that the difference between $(a(x_1),b(x_1))$
and $(a(x_*),b(x_*))$ is large. This phenomenon can be particularly relevant if $E \sim \frac{1}{\eps \ln \frac{1}{\eps}}$ and $E \approx a+b$, since then 
the characteristic scale in the Lotka-Volterra-regime is of order one up to logarithmic corrections. Therefore  $|x_1-x_*|$ might be much larger than this scale. 
However, we will prove an estimate that shows that the change of $(a,b)$ in the interval $[x_*,x_1]$ is negligible in the part of the Lotka-Volterra cycle
contained in   $\{a \leq 1\} \cup \{b \leq 1\}$. Indeed, the first equation \eqref{leadingorderode} implies that
$  \frac{da}{dx} \geq -\eps a$ and hence, using \eqref{X1},  $G(x_*) \approx \frac{\eps^2}{4} a(x_*)b(x_*)$ and $a(x_*) b(x_*) \geq C e^{-E} $, we find
\[
 a(x_1) \geq a(x_*) e^{\eps \ln G(x_*)} = a(x_*) G(x_*)^{\eps}\geq a(x_*) e^{\eps \ln (C\eps^2)} e^{-\eps E} \geq \frac{1}{2} a(x_*)\,. 
\]
Similarly, since $a(x_*), b(x_*) \leq 2E$, it follows that $b(x_1) \leq 2 b(x_*)$. In particular $a$ and $b$ cannot make more than one cycle in 
the Lotka-Volterra phase plane.

In order to describe the regime when $Q$ becomes of order one, we introduce  the new variable $t=x-x_1$. Then 
equation \eqref{Masympt} implies the matching condition
\begin{equation}
 \label{Nmatching}
 Q(t) = e^{t}  \qquad \mbox{ as } t \to -\infty\,.
\end{equation}
Denote by
$ A_1:=A(x_1)$, $B_1:=B(x_1)$ and $P_1:= 4A_1 B_1$
and define
$A=A_1\hat a$, $B=B_1 \hat b$ and $P = P_1 \hat p$.
Then we obtain 
\begin{align}
 \frac{d\hat a}{dt}&= - \frac{\eps}{2}\hat a + \frac{1}{2} \Big( 4 B_{1} \big(\hat a\hat b +  \hat p Q\big) + \frac{J[G]}{A_{1}}\Big)\,, \label{atau}\\
 \frac{d\hat b}{dt}&= \frac{\eps}{2} \hat b - \frac{1}{2} \Big( 4 A_{1}\big ( \hat a \hat b + \hat p Q \big)+ \frac{J[G]}{B_{1}}\Big)\,,\label{btau}\\
 \frac{d\hat p}{dt}&= -\ell + \Big(  \hat a \hat b + \hat pQ + \frac{J[G]}{P_{1}}\Big)\,, \label{elltau}\\
 \frac{dQ}{dt}&= Q- \big( P_1\big(\hat a \hat b + \hat  p Q \big) + J[G]\big)\,.\label{Ntau}
\end{align}
Using Lemma \ref{L.Gapprox} we can check that the terms involving $J[G]$ can be neglected in \eqref{atau}-\eqref{Ntau}. Since furthermore $A_1,B_1$ 
and $P_1$ are small, we obtain for $t $ of order one the 
following approximation of  \eqref{atau}-\eqref{Ntau}
 \begin{equation}\label{abtau}
 \frac{d\hat a}{dt}= 
 \frac{d\hat b}{dt}= 0  \,, \qquad 
 \frac{d\hat p}{dt}= -\hat p + \hat a \hat b + \hat p Q \,\qquad \mbox{ and } \qquad 
 \frac{dQ}{dt}= Q \,.
\end{equation}
Using the matching condition \eqref{Nmatching} we obtain $\hat a\equiv 1, \hat b \equiv 1$,
\begin{equation}\label{lformula}
\hat p =  e^{Q(t) -t} \Big( 1-e^{-Q(t)}\Big) \qquad \mbox{ and } \qquad Q(t) = e^{t}\,.
\end{equation}
Assume now that for large $t$ we have that 
$\hat p Q \gg \hat a \hat b$. We can then approximate \eqref{atau} and \eqref{elltau}, assuming for the moment that the nonlocal 
terms can still be neglected, by
\begin{equation}\label{abtaunew}
 \frac{d\hat a}{dt} = -\frac{\eps}{2}\hat a + 2 B_1 \hat p Q \,, \qquad \frac{d \hat b}{dt} = \frac{\eps}{2}\hat b - 2A_1 \hat p Q
 \qquad \mbox{ and } \qquad \frac{d \hat p}{dt} = \hat p (Q-1)\,.
 \end{equation}
The equation for $\hat b$ in \eqref{abtaunew} implies that 
\[
\hat  b(t) = e^{\frac{\eps}{2}t} - 2A_1\int_0^{t} e^{\frac{\eps}{2}(t-s)} (\hat p Q)(s)\,dx \approx  e^{\frac{\eps}{2}t} -
2A_1\int_0^{t} e^{e^{s}}\,ds
 \approx  e^{\frac{\eps}{2}t} - 2A_1 e^{-t} e^{e^{t}}\,.
\]
Then 
 $\hat b(t_{\eps})=0$ with $t_{\eps}$ as
 \begin{equation}\label{taueps}
e^{ e^{t_{\eps}} } e^{-t_{\eps}} \sim  \frac{1}{2A_1}\,,\qquad \mbox{ that is } \; t_{\eps} \sim \ln \Big(\ln \Big( \frac{1}{2A_1}\Big) \Big)\,.
\end{equation}
Since $A_1 \geq C e^{-E}$ we also have $t_{\eps} \leq \ln E \leq \ln \big(\frac{1}{\eps}\big)$. 
Notice that this justifies a posteriori  that one can neglect the term $\frac{\eps}{2}\hat b$ in the equation for $b$ even though $A_1$ might be very small.  

We now check that the assumption $\hat a \hat b \ll \hat p Q$ indeed holds for all $t \in [0,t_{\eps})$, which is not clear a priori given 
that $B_1$ could be much larger than $A_1$ and this could
produce a large growth of $\hat a$ in this interval.
Due to the first  equation in \eqref{abtaunew} we have
$\hat a (t) \leq 1+2B_1 \int_0^{t} (\hat p Q)(s)\,ds$ and the third equation in \eqref{abtaunew} give $\hat p (t) = e^{Q(t)-t}$.
This  gives $\hat a(t) \leq C B_1 \hat p (t)$, whence, since $\hat b$ is bounded,  $\hat a(t) \hat b(t) \leq CB_1\hat p(t)$.
Since $Q(t) \geq 1$ and since
$B_1 \ll 1$ we indeed obtain $\hat a\hat b \ll \hat p Q$.

In order to show the consistency of the whole approximation, it remains to check that the 
nonlocal term $J[G]$ remains small.
This is feasible, but requires some tedious estimates. In order not to interrupt too much the flow of the argument, we just state the main result.

The key remark  in order to compute the variations of $J[G]$ is that $\eps \hat p Q$ can be approximated by a mollification of the Dirac mass as follows
\begin{equation}\label{ellNapprox}
\begin{split}
 2A_1  \hat p Q &\approx 2A_1 e^{Q(t)}\big(1-e^{-Q(t)}\big) \approx  2A_1 e^{ e^{t}}\\
 &=
 2A_1 e^{e^{t_{\eps}} e^{t-t_{\eps}}}
 =2A_1 e^{e^{t_{\eps}} [e^{t-t_{\eps}}-1]} e^{e^{t_{\eps}}}
\approx \frac{1}{\delta_{\eps}}  e^{\frac{1}{\delta_{\eps}} (t-t_{\eps})}\,
\end{split}
\end{equation}
where we used the definition of $t_{\eps}$ in \eqref{taueps} and where 
 $\delta_{\eps}= \frac{1}{e^{t_{\eps}}} \approx \ln \frac{1}{2A_1}$. 
 
 By estimating carefully all the terms in $J[G]$ and using the previously obtained asymptotics we can obtain by lengthy but otherwise rather straightforward
 estimates that for sufficiently large $R>0$
 \begin{equation}\label{Jbound}
J[G](x)  \leq B_{1}\Big(  A_{1}+\frac{e^{\frac
{1}{\delta_{\varepsilon}}e^{x-x_{1}-t_{\varepsilon}}}}{\delta_{\varepsilon
}}\chi_{\{  t_{\varepsilon}-R\leq x-x_{1}\leq t_{\varepsilon}\}  }\Big)  \qquad \mbox{ for } \; x_{1}\leq x\leq x_{1}+t_{\varepsilon}\,.
\end{equation}

We need to compare the size of this nonlocal term with the main terms in the
equations  (\ref{atau})-(\ref{Ntau}) and check that the
contribution of this nonlocal term is small. This requires the inequalities:%
\begin{equation}
\frac{J[G]  }{A_{1}}\ll B_{1}\hat p Q\,,\qquad \frac{J[G]  }{B_{1}} \ll A_{1}\hat p Q\,, \qquad \frac{J[G] }{P_{1}}\ll\hat p Q\ \label{NonLocApp}%
\end{equation}
which are all equivalent. Moreover, the same inequality implies also that the term
$J[G]  $ in (\ref{Ntau}) gives a negligible contribution
compared with $P_{1}\hat p Q$ and it is possible to approximate this equation by
 an ODE.

In order to derive (\ref{NonLocApp}) we notice that (\ref{lformula}) and
(\ref{ellNapprox}) imply
\[
2A_{1}\hat p Q\geq C_{1}\Big(  A_{1}+\frac{e^{\frac{1}{\delta_{\varepsilon}%
}e^{x-x_{1}-t_{\varepsilon}}}}{\delta_{\varepsilon}}\chi_{\{
t_{\varepsilon}-R\leq x-x_{1}\leq t_{\varepsilon}\}  }\Big)
\]
for some $C_1>0$ and for $x_{1}\leq x\leq x_{1}+t_{\varepsilon}.$ Then (\ref{NonLocApp}) 
follows from \eqref{Jbound}.

We also remark that in the estimates of $J[G]$ we do not use Lemma \ref{L.Gapprox} due to the fast growth of $G$ that implies that the second inequality in \eqref{G1} fails.

\subsection{Instability is triggered for values of $x$ such that $E\sim\frac{1}{\varepsilon}.$}
\label{Ss.kineticinstability}

We now consider the case in which the instability is triggered at  times in which the energy $E$ defined in (\ref{energydef}) is of order $\frac
{1}{\varepsilon}.$ We recall that for these values of $E$ there are some values of $x$ for which we describe the solution $G,A,B,P,Q$ by 
solutions of the integro-differential equation (\ref{ODE1})-(\ref{Ident}) with $\varepsilon=0$ (cf. Subsection \ref{Ss.kinetic}),
while for other values 
 (\ref{ODE1})-(\ref{Ident}) is approximated by  a
system of ODEs (cf. Subsections \ref{Ss.ode1}, \ref{Ss.absmall} and \ref{Ss.ode2}). We will denote the first
set of values of $x$, where $G$ is of order one,  as values of Type I and the second set, where $G\ll 1$,
as values of Type II.
For $x$
of Type II the values of the functions $G,A,B,P,Q$ change in a scale 
much longer than one. Then, the triggering of the instabilities at some
$x=x^{\ast}$ can be described by the asymptotic behaviour for $x<x^{\ast}$ and  $|x-x^{\ast}|\gg 1$.
For values  $x$ of Type I the characteristic length scale for the
development of the instability given by the equation for $Q$ in
(\ref{ODE2}) is the same as  for the
variation of $G$. Then  we cannot obtain a definition like (\ref{Masympt})
for the value of $x^{\ast}$ in which the instability is triggered. 
In this case we need to study  in detail the solutions of the Volterra-like problem
with additive kernel \eqref{A1}-\eqref{A3}.

\subsubsection{Additional solutions of  problem \eqref{A1}-\eqref{A3}.}

In  Lemma \ref{L.kinetic} we obtained a family of solutions to  \eqref{A1}-\eqref{A3}, that correspond to solutions of the coagulation equation.
These are however not the only solutions and in the following proposition, which is proved in the Appendix, we construct additional solutions with different 
higher order asymptotics as $x \to -\infty$.

\begin{prop}
\label{P.VoltLikePb}

Suppose that for $\rho\in(0,2]  $ we denote by
$A_{\rho},B_{\rho},P_{\rho},Q_{\rho}$ and $G_{\rho}$ the solution in Lemma
\ref{L.kinetic} that are  characterized by the asymptotics
(\ref{Grhominus}). 
For any $\alpha \in \R$ we define
\begin{equation}\label{Galphadef}
 G_{\rho}^{\alpha}(x) = G_{\rho}(x) e^{\alpha e^x}
\end{equation}
and, with $B_{\infty}:=B_{\rho}(-\infty)$,
\begin{equation}\label{alphadefs}
\begin{split}
 A_{\rho}^{\alpha}(x) = \frac{1}{2}\int_{-\infty}^x G_{\rho}^{\alpha}(y)\,dy\, ,& \qquad  B_{\rho}^{\alpha}(x) = B_{\infty}- 
 \frac{1}{2} \int_{-\infty}^{0} G_{\rho}^{\alpha}(y)\,dy\\
 P_{\rho}^{\alpha}(x) = \int_{-\infty}^x G_{\rho}^{\alpha}(y)e^{y-x}\,dy\,,& \qquad  Q_{\rho}^{\alpha}(x) = Q_{\rho}(x)e^{\alpha e^x}+ \alpha e^{x} \Big(
 1-\int_{-\infty}^x e^{\alpha e^y}Q_{\rho}(y)\,dy\Big)\,.
 \end{split}
 \end{equation}
Then $A_{\rho}^{\alpha}, \dots, G_{\rho}^{\alpha}$ solve \eqref{A1}-\eqref{A3}.

Furthermore, for $\rho \in (0,1]$, if $\alpha<0$, we obtain that $Q_{\rho}^{\alpha}(x)$ becomes negative for some $x\in \R$,
while if $\alpha>0$ then $B_{\rho}^{\alpha}$ becomes negative for some $x \in \R$.
\end{prop}

\begin{rem}
 The solutions constructed in Proposition \ref{P.VoltLikePb} are the only solutions of \eqref{A1}-\eqref{A3} with the asymptotics $G(x) \to 0$ as $x \to -\infty$.
Indeed, suppose that $G >0$,  $G(x)\sim e^{\beta x}$  and $B(x)\to B_{\infty}$ as $x \to -\infty$.
 We can assume that $B_{\infty}>0$ since otherwise $A \to \infty$ and then also $G \to \infty$ as $x \to -\infty$. 
 The equations in  \eqref{A2} imply that $ Q\leq Ce^{\beta x}$, $P \leq Ce^{\beta x}$ 
 and $J[G](x) \leq Ce^{2\beta x}$. Then \eqref{A3} implies that $G(x) \sim 4 B_{\infty} A(x)$ and \eqref{A1} implies that 
 $A(x)=c_1 e^{2 B_{\infty} x}$ as $x \to -\infty$. 
Thus, $G(x) \sim k_1 e^{2 B_{\infty} x}$ as $x \to -\infty$ with $\beta = 2B_{\infty}$.
 Using this asymptotic formula we can compute also the asymptotics of $A,B,P,Q$ and $J[G]$ as $x \to -\infty$ to obtain
 $G(x) \sim K_1e^{\beta x} + k_2 e^{2\beta x} + \cdots +k_m e^{m\beta x} + \cdots $ as $x \to -\infty$ with $m\beta \leq 1+\beta< (m+1)\beta$. 
 (In some resonance cases the terms $e^{k\beta x}$ might
 have to be replaced by $xe^{k\beta x}$.)
 
 However, computing the next order in this asymptotics we can include an additional parameter $\bar k$ from the second equation in \eqref{A2}, since we have 
 $Q(x)=\nu_1 e^{\beta x} + \nu_2 e^{2 \beta x} + \cdots + \nu_{m}e^{m\beta x} + \bar k e^{x}$ as $x \to -\infty$. Plugging this into \eqref{A3} 
 gives
 \begin{equation}\label{rem1}
  G(x) \sim k_1e^{\beta x} + k_2 e^{2\beta x} + \cdots +k_m e^{m\beta x} + C\bar k e^{(1+\beta)x} + \cdots \qquad \mbox{ as } x \to -\infty
  \end{equation}
  with the free parameters $k_1$ and $\bar k$, such by varying $\alpha$ in Proposition \ref{P.VoltLikePb} one indeed obtains  all solutions.
 \end{rem}
 
Proposition \ref{P.VoltLikePb} implies 
that we must have $\alpha=0$ at the beginning of each of the ranges of $x$ for
which the solutions of (\ref{ODE1})-(\ref{Ident}) are approximately described by
 the solutions of \eqref{A1}-\eqref{A3}.

\subsubsection{ODE regimes with $E\approx\frac{1}{\varepsilon}.$}
\label{Ss.odeinstability}

We now examine the case in which the instability for $Q$ in the second
equation of (\ref{ODE2}) is triggered for values of $x$ for which $E$ is
of order $\frac{1}{\varepsilon}$ but $G$ is small. For those ranges of $x$ the
dynamics of $G(x) $ is described, before the instability is
triggered, as in Subsections \ref{Ss.ode1}, \ref{Ss.absmall} and \ref{Ss.ode2}. In all those cases the
functions $A,B,P,Q$ and $G$ are approximated by  ODEs.
Given that the characteristic time scale for all these functions is much
larger than one and the instability associated to $Q(x)  $ is
triggered in times of order one, we can describe the triggering
by (\ref{Masympt}).
We can describe the evolution 
after the instability is
triggered  similarly as  in Subsection \ref{Ss.fastinstability}. By assumption $G(x_{\ast})  $ is
small and due to Lemma \ref{L.Gapprox} we can approximate (\ref{ODE1})-(\ref{Ident}) by  a system of ODEs. Moreover, as long as
$G(x)  $ remains small we can neglect the contribution of the term
$P(x)Q(x) $ in \eqref{Ident}. Then, 
we can approximate the evolution of the functions $A,B$ and $G$
 using the ODEs in Subsections \ref{Ss.ode1}, \ref{Ss.absmall} and \ref{Ss.ode2}.
This approximation is valid as long as $P(x)Q(x)  $ is small compared to $G(x)$
and  the evolution of $P(x)$ and $Q(x)$  is described by  (\ref{ODE2}).
 If the sign in  (\ref{Masympt}) is a minus,  we can then conclude exactly as in Subsection \ref{Ss.fastinstability} 
 that $Q$ becomes negative.

Suppose now that the sign in (\ref{Masympt}) is a plus sign. Using the
same arguments as above, it follows that as long as $P(x)Q(x) \ll G(x)  $ we can
approximate $Q(x) $ as $G(x_{\ast})  (  1+e^{(  1-\frac{\varepsilon}{2}) (  x-x_{\ast})  })$. We then
define $x_{1}$ by  (\ref{X1}). Notice that
as with the fast instability described in Subsection \ref{Ss.fastinstability} we
might have that $x_{1}-x_{\ast}$ is very large and therefore we might have relevant
changes of the values of $A$ and $B$ during
this phase. However,  as long as we can use the ODE
approximations in Subsections \ref{Ss.ode1}, \ref{Ss.absmall} and \ref{Ss.ode2} we can deduce, arguing as in
Subsection \ref{Ss.fastinstability}, that either $a(x_1)  \geq\frac{a(x_{\ast})  }{2}$ or $b(x_{1}) \leq2b(x_{\ast})$ and conclude the argument as before.

\subsubsection{Fast instability in the kinetic regime.}\label{Ss.instkinetic}

A  difference to the argument in Subsection \ref{Ss.fastinstability}
is that in principle it might be possible for $G(x)  $ to reach
values of order one during the interval $[x_{\ast},x_{1}]$.
Then we recall the definition of the sequence $\{  \hat{x}_{n}\}  $ in Subsection \ref{Ss.kinetic}
and that those are 
 the characteristic times for which we have the approximation
$G(x)  \simeq G(x-\hat{x}_{n})  .$ 
For $x<\hat x_n$ and $|x-\hat x_n|\gg 1$ we can match the
asymptotics of the solutions for $x-x_{\ast}\gg 1$. With $t=x-\hat{x}_{n}$ 
we have the approximation
\[
Q\simeq G(  x_{\ast})  e^{(  1-\frac{\varepsilon}{2})
(x-x_{\ast})  }=G(x_{\ast})  e^{(1-\frac{\varepsilon}{2}) (  \hat{x}_{n}-x_{\ast})
}e^{(  1-\frac{\varepsilon}{2})  (x-\hat{x}_{n})
}=G(x_{\ast})  e^{(  1-\frac{\varepsilon}{2})(\hat{x}_{n}-x_{\ast})  }e^{(  1-\frac{\varepsilon}{2})  t}
\]
for $t<0$ and $|t|\gg 1$. By assumption $G(x_{\ast})  e^{(  1-\frac{\varepsilon}%
{2})(\hat{x}_{n}-x_{\ast})  }$ is at most of order one.
If we have that $G(  x_{\ast})  e^{(  1-\frac{\varepsilon}{2})  (  \hat{x}_{n}-x_{\ast})  }$ is of order one we would
have that $G$ can be approximated for $t<0,\;\vert t\vert $
large by a solution of the Volterra-like problem for the additive kernel (cf.
\eqref{A1}-\eqref{A3}) with matching condition
$G(t)  \sim G_{\rho_n}^{\alpha }(t)$ as $t \to -\infty$
where $\alpha= G(x_{\ast})  e^{(  1-\frac{\varepsilon}{2}) (  \hat{x}_{n}-x_{\ast})  }.$ If $\alpha>0$ is of
order one it then follows that $B_{\rho_n}^{\alpha}(x)  $ becomes negative for a
finite value. Alternatively, we might have $\alpha \rightarrow0.$ In that
case we can approximate $A(x) $, $B(x)  $, $P(x)$, $Q(x)$ and $G(x)  $ using the functions
$A_{\rho_n}^{\alpha}, \dots, G_{\rho_n}^{\alpha}$ with small $\alpha$.
Notice that for $t$ of order
one we cannot approximate any longer $Q$ by  the second equation in
(\ref{Nminus}) but  $Q$ is close to $Q_{\rho_n}(t)  $ for $t$ of order one. However, using (\ref{alphadefs}) as well as
the fact that $x=e^{t}$ we obtain the following asymptotics for $\alpha \to 0$ and then $t \gg 1$
\begin{equation}
Q\sim Q_{\rho_n}(t)  +\alpha \big (  1-\int_{-\infty}^{\infty}%
Q_{\rho_n}(y)\big)  e^{t}\sim G_{\rho_n}(t)  +\alpha (1-M)  e^{t}\,, \label{T3E5}%
\end{equation}
where we use the fact that for large $t$ we have $Q_{\rho_n}(t)  \sim G_{\rho_n}(t)  $ and that 
$\int_{-\infty}^{\infty}Q_{\rho_n}(y)dy=\int_{-\infty}^{\infty}G_{\rho_n}(y)dy=\frac{\rho_n}{1+\rho_n}<1$.
 Notice that (\ref{T3E5})
yields an interesting connection condition for the coefficient describing the
exponential separation of $Q$ from $G.$ Indeed,  during the
kinetic regime this separation takes place according to $\alpha e^{t}$
for $t\rightarrow-\infty$ and it becomes $\alpha (1-M)
e^{t}$ as $t\rightarrow\infty.$ For sufficiently large  $t$ 
the term $\alpha (1-M)  e^{t}$ becomes larger than $G_{\rho_n}(t)$.
 Using the asymptotics of $G_{\rho_n}(t)$ (cf. (\ref{Grhoplus})) it follows that this happens for $e^{(1+\rho_n)  t}\approx\frac{1}{\alpha}$
and $G_{\rho_n}(t)  $
is of order $\alpha^{\frac{\rho_n}{1+\rho_n}}.$ 
For large $t$  we enter the regime where $G$ is
described as in Subsection \ref{Ss.ode1}. In particular $2 A(t)  $
approaches $M$ for large $t$ and arguing as in Subsection \ref{Ss.overview} it follows
that $A$ cannot decrease to values smaller than $\frac{M}{4}$ without arriving
to values of $x=x_{1}$ where (\ref{X1}) takes place. 
We can then argue as in the rest of Subsection \ref{Ss.overview} to show that $B$ changes
sign. Indeed, the arguments used there only require that $E$ is bounded by
$\frac{C}{\varepsilon}.$

It is relevant to remark  that the amplitude
of the exponentially growing perturbation implicitly contained in the equation
for $Q$ (cf. \eqref{ODE2}) cannot become arbitrarily small as
$x\rightarrow\infty$ in spite of the fact that it is reduced by a factor
$(1{-}M)  $ each time that the solution passes by one of the
intervals $[x_{n},\bar{x}_{n}]  $ where it is described by  the solutions  described in
Lemma \ref{L.kinetic}. Indeed, during the intervals $(\bar{x}_{n},x_{n+1})  $ in which the function $G(x)$ is
described by means of the ODE approximations described in Subsections \ref{Ss.ode1}-\ref{Ss.ode2}  the perturbation of  $Q-G$ would
continue growing exponentially and the length of the intervals $(\bar{x}_{n},x_{n+1})  $ is very large of order $\frac{1}{\varepsilon}$). 
This growth is much larger than the decrease by the factor $(1{-}M)$ during the kinetic regime.

\subsection{Summary of the analysis of the instabilities}
\label{Ss.summary}

 We have seen, using formal arguments, that an instability for $Q$ yields a change of sign, either of $Q$ or of $B$.
To this end we have considered several cases. Initially the
Volterra-like problem can be approximated by a Lotka-Volterra equation. In
such a case, ODE arguments show that instabilities for $Q$ yield changes of
sign, either for $Q$ or $B$.

 It is possible to compute explicit formulas for the solutions of the
Volterra-like problem associated to the additive kernel. The final conclusion
is the same. We can put an additional parameter $\alpha$ in the solution
as $x\rightarrow-\infty.$ If $\alpha<0$ we obtain that $Q$
 changes sign and it vanishes for finite $s.$ If $\alpha>0$ we obtain
that $B$ becomes negative for finite $x.$ 

 The most delicate case corresponds to the ODE regimes in the case when the energy is large.
that case there are different possibilities. In some cases $B$ or $Q$ change
sign quickly, because the integral terms are irrelevant and the problem can be
approximated by an integral equation. In other cases the solutions must be
approximated by the solutions of the additive kernel (the Volterra-like
problem) discussed before.

\section{Appendix}

\subsection{Proof of Lemma \ref{L.roots}}

We assume from now on that $\eps \leq \eps_0$ for sufficiently small $\eps_0$. Then we claim that there exists $R\geq 3$ that is independent of $\eps$ such that 
 all solutions of
\eqref{roots1} are contained in the set $B_{R}(0)  \cap\{\operatorname{Re}(\mu)  \geq0\}  $.

Indeed, if $|\mu| \geq R$ and $\mbox{Re}(\mu)>0$ we have, using 
Stirling's formula,
\begin{equation}\label{Tequation}
\Big\vert \frac{\Gamma( 1-\frac{\varepsilon}{2}+\mu)  }%
{\Gamma(1+\mu)  }\Big\vert  \leq2\Big\vert \frac{\left(  1-\frac{\varepsilon}{2}+\mu\right)^{(1-\frac{\varepsilon}{2}+\mu) }}{(1+\mu)^{(1+\mu)}}
\Big\vert=:2 \exp\big(T(\mu,\eps)\big)\,.
\end{equation}
where
\begin{align*}
T(\mu,\eps)
& =\operatorname{Re}\Big(  (  1-\frac{\varepsilon}{2}+\mu)\Big(  \ln( \vert 1-\frac{\varepsilon}{2}+\mu\vert )
+i\arg(  1-\frac{\varepsilon}{2}+\mu) \Big)\Big) \\
&  \qquad -\operatorname{Re}\Big( ( 1+\mu) \Big(  \ln(\vert 1+\mu\vert )  +i\arg(1+\mu)\Big)\Big)\,.
\end{align*}

Using that  $\left\vert \mu\right\vert \geq R\geq3$ 
we obtain 
\begin{equation}\label{Testimate}
 \begin{split}
T(\mu,\eps) &\leq \operatorname{Re}\Big( (1+\mu) \Big(  \ln(\vert 1-\frac{\varepsilon}{2}+\mu\vert)  
+i\arg(1-\frac{\varepsilon}{2}+\mu)\Big)\Big)  \\
& \qquad  -\operatorname{Re}\Big((  1+\mu)\Big(  \ln(\vert 1+\mu\vert )  +i\arg(1+\mu) 
\Big)\Big)\\
&  \qquad -\operatorname{Re}\Big(  (1+\mu)  \ln(\vert1+\mu\vert) \Big)  -\operatorname{Re}\Big(  i\Big(
1+\mu)  \arg(1+\mu)\Big) \\
& =\operatorname{Re}\Big(  (1+\mu)  \ln(  \frac{\vert 1-\frac{\varepsilon}{2}%
+\mu\vert }{\vert 1+\mu\vert })  \Big)  \qquad  +\operatorname{Re}\Big(  i(1+\mu) \Big( \arg(1-\frac{\varepsilon}{2}+\mu)  -  \arg( 1+\mu) \Big) \Big)\,.
\end{split}\end{equation}

We have
\begin{align*}
\ln\Big(  \frac{\vert 1-\frac{\varepsilon}{2}+\mu\vert}{\vert 1+\mu\vert }\Big)   & \leq2\Big\vert \Big\vert
\frac{1-\frac{\varepsilon}{2}+\mu}{1+\mu}\Big\vert -1\Big\vert   \leq\frac{2\varepsilon}{\vert 1+\mu\vert }
\end{align*}
and choosing  $\arg(\cdot) \in (-\pi,\pi)  $ we obtain
\[
\big\vert \arg(  1-\frac{\varepsilon}{2}+\mu)  -\arg(1+\mu)  \big\vert = \Big \vert \arg\big(
1-\frac{\varepsilon}{2}\frac{1}{1+\mu}\big)\Big \vert \leq\frac{\varepsilon}{\vert 1+\mu\vert }\,.
\]
Combining these last two inequalities with \eqref{Testimate} we find
$T(\mu,\eps) 
 \leq\frac{3\varepsilon}{\vert 1+\mu\vert }\vert1+\mu\vert \leq3\varepsilon
$ whence, using \eqref{Tequation}, it follows that $\big\vert \frac{\Gamma(  1-\frac{\varepsilon}{2}+\mu)  }{\Gamma(  1+\mu)  }\big\vert \leq3$.
Similarly
we obtain
$\big\vert \frac{\Gamma(\frac{\varepsilon}{2}+\mu)  }{\Gamma(1+\mu)  }\big\vert \leq3$.

Thus, if $\mu$ solves \eqref{roots1}, we have
\[
 \frac{\Gamma(1-\frac{\varepsilon}{2})  }{(  1-\frac{\varepsilon}{2})  \frac{\varepsilon}{2}} 
 =\frac{\Gamma(1-\frac{\varepsilon}{2}+\mu)  }{\Gamma(  1+\mu)  }\frac{(1-\mu)  }{(1-\frac{\varepsilon}{2}) (  \frac{\varepsilon}{2}-\mu)  }
+\frac{\Gamma(  \frac{\varepsilon}{2}+\mu)  }{\Gamma(1+\mu)  }\Gamma\big(1-\frac{\varepsilon}{2}\big)  \frac{(1-\mu)  }{\Gamma(  \frac{\varepsilon}{2})
\frac{\varepsilon}{2}(  1-\frac{\varepsilon}{2}-\mu)  }
\]
and using the above inequalities and the fact that $\lim_{z\rightarrow0}  \Gamma(z)
z  =1$ we obtain that for any root $\mu$ of \eqref{roots1} we have
\[
\frac{\Gamma(  1-\frac{\varepsilon}{2})  }{(  1-\frac {\varepsilon}{2})  \frac{\varepsilon}{2}}\leq3\Big\vert \frac{(1-\mu)  }{(  1-\frac{\varepsilon}{2})
(\frac {\varepsilon}{2}-\mu)  }\Big\vert +4\Gamma\big(1-\frac{\varepsilon}{2}\big)\Big \vert \frac{(1-\mu)  }{(1-\frac{\varepsilon}{2}-\mu)  }\Big\vert\,.
\]

The right-hand side of this inequality is bounded by a constant independent of
$\varepsilon$ if $\vert \mu\vert \geq R\geq3.$ However, the
left-hand side diverges as $\varepsilon\rightarrow0.$ This gives a
contradiction. Therefore, all  solutions of \eqref{roots1} satisfy
$\vert \mu\vert <R.$

We now rewrite \eqref{roots1} as 
\begin{equation}\label{roots2}
0=-\frac{\Gamma(1-\frac{\varepsilon}{2})\Gamma(1+\mu)  }{(1-\frac{\varepsilon}{2})
\frac{\varepsilon}{2}}  +\frac{\Gamma(1-\frac{\varepsilon}{2}+\mu)  (1-\mu)  }{(1-\frac{\varepsilon}{2}) 
(  \frac{\varepsilon}{2}-\mu)  } +\frac{\Gamma(\frac{\varepsilon}{2}+\mu)\Gamma(1-\frac{\varepsilon}{2})  (1-\mu)  }{\Gamma(\frac{\varepsilon}{2})  
\frac{\varepsilon}{2}(1-\frac{\varepsilon}{2}-\mu)  }=:\omega(\mu,\eps)%
\end{equation}

Using that $\Gamma(\frac{\varepsilon}{2})  \frac{\varepsilon}%
{2}\rightarrow1$ as $\varepsilon\rightarrow0$ it  follows that the
right-hand side of \eqref{roots2} is bounded in the set $\{\operatorname{Re}(\mu)  \geq0,\ \vert \mu\vert \leq R,\ \min\{ \vert \mu-1\vert ,\vert \mu\vert\} 
\geq\delta\}  $ for each $\delta>0$ if we take $\varepsilon$
sufficiently small (depending on $\delta$). However, since $\vert \Gamma(1+\mu)\vert \geq C_{\ast}>0$ in the set 
$\{\operatorname{Re}(\mu)  \geq0,\ \vert \mu\vert \leq R\}  $ it  follows that the left hand side tends to infinity as
$\varepsilon\rightarrow0$ uniformly in this set. Then, all roots of
\eqref{roots1} in $\{\operatorname{Re}(\mu)\geq0\}  $ are contained in the set $\{\operatorname{Re}(\mu)  \geq0,\ \vert \mu\vert \leq R,\ \min\{ \vert
\mu-1\vert ,\vert \mu\vert \}  <\delta\}  .$ We
can compute the asymptotics of the roots of \eqref{roots1}
approximating $\omega(\mu,\eps)$ in \eqref{roots2} by suitable comparison functions and  Rouche's
Theorem.

In the region $\vert \mu\vert \leq\delta$ with small $\delta$,  we use as comparison function
$ \omega_1(\mu,\eps):= -\frac{2}{\eps}+ \frac{\eps}{( \frac{\eps^2}{4}-\mu^2)} + \frac{\mu}{( \frac{\eps}{2}-\mu)}$,
while
in the region $|\mu-1| \leq \delta$ we use
$ \omega_2(\mu,\eps):=
 -\frac{2}{\varepsilon}  +(  1-\mu)  \frac{\frac{\varepsilon}%
{2}+\mu}{(  1-\frac{\varepsilon}{2}-\mu)  }$.
Using the asymptotics of the function $\Gamma(z)$ for $z \sim 1$ and $z \sim 0$ we obtain \eqref{muone}.

\subsection{Proof of Proposition \ref{P.VoltLikePb}}

We first state the following auxiliary lemma.
\begin{lemma}
\label{IntODEAux}Suppose that $F,W\in L^{1}(0,R)  $ satisfy
\[
W(\xi)  -\lambda\int_{0}^{\xi}W(\eta)  d\eta=F(\xi)  \qquad \mbox{ a.e.  in } (0,R)
\]
for some $\lambda\in\mathbb{R}$. Then
\begin{equation}
W(\xi)  =AF(\xi)  :=\lambda e^{\lambda \xi}\int
_{0}^{\xi}e^{-\lambda \eta}F(\eta)  d\eta+F(\xi) \qquad \mbox{ a.e.  in } (0,R)\,. \label{T2E1}%
\end{equation}
\end{lemma}
\begin{proof}
We define $H (\xi)  =\int_{0}^{\xi}W(\eta)  d\eta.$ The
function $H$ is absolutely continuous and satisfies
$\frac{dH(\eta)  }{d\eta}-\lambda H(\eta)  =F(\eta)$ and $H(0)=0$.
Therefore $H(\xi)  =e^{\lambda \xi}\int_{0}^{\xi}e^{-\lambda \eta}F(\eta)  d\eta$ for $\xi\in(0,R)  ,$ whence (\ref{T2E1}) follows.
\end{proof}

For the proof of Proposition \ref{P.VoltLikePb} it is convenient to go over to the variable $\xi =e^x$.
For notational convenience we neglect the index $\rho$ throughout this proof.

We
first define functions $a$, $b$, $p$, $q$ and $h$ for
$\xi>0$ by 
\begin{equation}
a(\xi)  =A(x)  \,,\quad b(\xi)  =B(x)\,,\quad p(\xi)=P(x)\,,\quad
q(\xi)=Q(x)  \,,\quad  \xi^{2}h(\xi)  =G(x)\,.  \label{T1E1}%
\end{equation}

We then define functions $a^{\alpha}$, $b^{\alpha}$, $p^{\alpha}$, $q^{\alpha}$ and $G^{\alpha}$ as
\begin{equation}\label{T1E2}
 \begin{split}
h^{\alpha }(\xi)   &  =h(\xi)  e^{\alpha \xi}\,, \qquad  p^{\alpha  }(\xi)  =\frac{1}{\xi}\int_{0}^{\xi}\eta^{2}h^{\alpha}(\eta)  d\eta\\
a^{\alpha}(\xi)&  =\frac 1 2 \int_{0}^{\xi}\eta h^{\alpha  }(\eta)  d\eta\,, \qquad b^{\alpha  }(\xi)  =B_{\infty} -\frac 1 2 \int_{0}^{\xi}\eta h^{\alpha}(\eta)  d\eta\\
q^{\alpha}(\xi)   &  =q(\xi)  e^{\alpha \xi}+\alpha \xi\Big(  1-\int_{0}^{\xi}e^{\alpha \eta}\frac{q(\eta)  }{\eta}d\eta\Big) \,.
\end{split}
\end{equation}
Then it holds 
\begin{equation}
\begin{split}
A^{\alpha}(x)=a^{\alpha  }(\xi)&\qquad  B^{\alpha}(x)  =b^{\alpha  }(\xi)\,,\qquad   P^{\alpha  }(x)  =p^{\alpha  }(\xi)\,\\
Q^{\alpha}(x)=q^{\alpha  }(\xi)& \qquad  G^{\alpha}(x) =\xi^{2}h^{\alpha}(\xi)\,,
\end{split}\label{T1E7}%
\end{equation}
with $A^{\alpha}, \cdots, G^{\alpha}$ as in \eqref{alphadefs}.

Notice that the asymptotics of the functions $A, \cdots G$
 in Lemma \ref{L.kinetic} imply that all the
integrals appearing in (\ref{T1E2}) are well defined. In particular notice
that $0\leq h(\xi)  \leq C\xi^{b-2}$ and $0\leq
\frac{q(\xi)  }{\xi}\leq C\xi^{b-1}$ for $\xi\leq1$ with
$b>0.$ This implies that all the integrals appearing on the right-hand
side of (\ref{T1E2}) are convergent.
One slight complication is due to the fact that $h(\xi)$ is not integrable and thus we cannot write directly
$\frac{q(\xi) }{\xi}=-\int_{0}^{\xi}h(\eta)d\eta.$ The
singularity of $h(\xi)  $ is like $\frac{1}{\xi^{2-\sigma}}$ with
$\sigma>0$ such that $\frac{q(\xi)}{\xi}$ is integrable and $q(\xi) \sim \xi^b$ as $\xi \to 0$.

It is worth  noticing that the last equation in (\ref{T1E2}) can due to Lemma \ref{IntODEAux} be inverted to obtain
\begin{equation}\label{InvForm}
q(\xi)  =n^{\alpha}(\xi)e^{-\alpha \xi} + \alpha \xi \int_0^\xi \frac{q^{\alpha}(\eta)}{\eta} e^{- \alpha \eta}d\eta-\alpha.
\end{equation}

Using the change of variables $\xi=e^{x}$ and (\ref{T1E1}) in \eqref{A1}-\eqref{A3} we obtain that the
functions $a$, $b$, $p$, $q$ and $h$ solve
\begin{align}
\frac{da}{d\xi}  &  =\frac 1 2 \xi h(\xi)  \,, \qquad  \frac{db}{d\xi}=-\frac 1 2 \xi h(\xi) \label{T1E3}\\
\frac{d}{dx}(\xi p(\xi))   &  =\xi^{2}h(\xi)\,,\qquad  \frac{d}{d\xi}\Big(  \frac{q(\xi)  }{\xi}\Big) =-h( \xi)  \label{T1E4}\\
\xi^{2}h(\xi)   &  =4 a(\xi)  b(\xi)+ p(\xi)q(\xi)  +j[h]  (\xi)
\label{T1E5}
\end{align}
with $j[h](\xi)    =\int_{0}^{\xi}\eta h(\eta) d\eta\int_{\xi-\eta}^{\xi}h(\zeta)(  \eta+\zeta)  d\zeta$.

Moreover, the definition of $a^{\alpha}$ and $b^{\alpha}$  in (\ref{T1E2}) implies that these two
functions solve (\ref{T1E3}) with $h=h^{\alpha}$. Similarly $p^{\alpha}$ and  $h^{\alpha}$
 solve the first equation of (\ref{T1E4}). It remains to check that $h^{\alpha}$ and $q^{\alpha}$ 
 satisfy the second equation of (\ref{T1E4}) and (\ref{T1E5}). In order to simplify the
notation we define the functions  $\varphi(\xi)
=\frac{q(\xi)  }{\xi}$ and $\varphi^{\alpha}(\xi)= \frac{q^{\alpha}(\xi)}{\xi}$.
 Using  \eqref{InvForm} we obtain 
\begin{equation}
\varphi (\xi)  =\varphi^{\alpha}(\xi)e^{-\alpha}\xi+ \alpha \int_0^\xi \varphi^{\alpha}(\eta)e^{-\alpha \eta}\,d\eta - \alpha
\label{T2E4}%
\end{equation}
and it is straightforward to check that $q^{\alpha}$ and $h^{\alpha}$   solve the second equation of (\ref{T1E4}).

We now check that (\ref{T1E5}) holds. To this end we first transform the equation to a more convenient form. Notice that both sets of
functions $a,\dots, h$, as well as $a^{\alpha}, \dots, h^{\alpha}$ satisfy $b (\xi)  =B_{\infty}-\int_{0}^{\xi}\eta h(\eta)d\eta$
and $p(\xi)  =\frac{1}{\xi}\int_{0}^{\xi}\eta^{2}h(\eta)  d\eta$.
Then (\ref{T1E5}) can be rewritten as
\begin{align}
\xi^{2}h(\xi)   &  =B_{\infty}\int_{0}^{\xi}\eta h(\eta)d\eta-\int_{0}^{\xi}\eta h(\eta)  d\eta\int_{0}^{\xi-\eta}\zeta h(\zeta)d\zeta
+\frac{q(\xi)}{\xi}  \int_{0}^{\xi}\eta^{2}h(\eta)  d\eta \label{T2E2}\\
& \qquad +\int_{0}^{\xi}\eta^{2}h(\eta)  d\eta\int_{\xi-\eta}^{\eta}h(\zeta)d\zeta\,.\nonumber
\end{align}

We now rewrite the last term, using the second equation of (\ref{T1E4}), as
\[
\int_{0}^{\xi}\eta^{2}h(\eta)  d\eta\int_{\xi-\eta}^{\xi}h(\zeta)\,d\zeta
=-\int_{0}^{\xi}\eta^{2}h(\eta)\Big(  \frac{q(\xi)  }{\xi}-\frac{q(\xi-\eta)}{\xi-\eta}\Big)d\eta\,.
\]
Plugging this in (\ref{T2E2}) we find, after some simplifications, 
 and using the definition of the function $\varphi(\xi)=\frac{q(\xi)}{\xi}$ that
\begin{equation}
\xi^{2}h( \xi)  =B_{\infty}\int_{0}^{\xi}\eta h(\eta)  d\eta-\int_{0}^{\xi}\eta h(\eta)  d\eta\int_{0}^{\xi-\eta}\zeta h(\zeta)  d\zeta+
\int_{0}^{\xi}\eta^{2}h(\eta)  \varphi(\xi-\eta)  d\eta\,. \label{T2E3}%
\end{equation}
Thus,  (\ref{T1E3})-(\ref{T1E5}) is equivalent to  (\ref{T1E3}), (\ref{T1E4}), (\ref{T2E3}).

We will now check that, if $\varphi$ and $h$ solve  (\ref{T2E3}), we have the same for $\varphi^{\alpha}$ and $h^{\alpha}$. 
Using the definitions in \eqref{T1E2} we obtain
\begin{equation}\label{T2E5}
 \begin{split}
& \xi^{2}h^{\alpha  }(xi)  e^{-\alpha \xi}+\alpha \int_{0}^{xi}\eta^{2}h^{\alpha}(\eta)   e^{-\alpha \eta}d\eta\\
  &  =B_{\infty}\int_{0}^{\xi}\eta h^{\alpha}(\eta)  e^{-\alpha \eta}d\eta-\int_{0}^{\xi}\eta h^{\alpha}(\eta) 
  e^{-\alpha \eta}d\eta\int_{0}^{\xi-\eta}\zeta h^{\alpha  }(\zeta)  e^{-\alpha \zeta}d\zeta
  \\
  &\qquad +e^{-\alpha \xi}\int_{0}^{\xi}\eta^{2}h^{\alpha}(\eta)  \varphi^{\alpha  }(\xi-\eta)  d\eta+\alpha \int_{0}^{\xi}\eta^{2}h^{\alpha}(\eta)
  e^{-\alpha \eta}d\eta\int_{0}^{\xi-\eta}\varphi^{\alpha}(\zeta)  e^{-\alpha \zeta}d\zeta\,.
\end{split}
\end{equation}

We can now apply Lemma \ref{IntODEAux} with $W(\xi)=\xi^2h^{\alpha}(\xi) e^{-\alpha \xi}$ to obtain 
\begin{equation}
\xi^{2}h^{\alpha}(\xi)e^{-\alpha \xi}=B_{\infty}AF_{1}(\xi)  -AF_{2}(\xi)  +AF_{3}(\xi)  +\alpha AF_{4}(\xi)  \,,\label{T2E6}%
\end{equation}
where $A$\ is the operator defined in (\ref{T2E1}) with $\lambda=-\alpha$
and
\begin{align*}
F_{1}(\xi)    =\int_{0}^{\xi}\eta h^{\alpha}(\eta) e^{-\alpha \eta}\,d\eta\,,\qquad &
F_{2}(\xi)     =\int_{0}^{\xi}yh^{\alpha}(\eta)e^{-\alpha \eta}\,d\eta 
\int_{0}^{\xi-\eta}\zeta h^{\alpha}(\zeta) e^{-\alpha \zeta}\,d\zeta\\
F_{3}(\xi)     =e^{-\alpha \xi}\int_{0}^{\xi}\eta^{2}h^{\alpha}(\eta)
  \varphi^{\alpha}(\xi-\eta)d\eta\,, \qquad & F_{4}(\xi)     =\int_{0}^{\xi}\eta^{2}h^{\alpha}(\eta) e^{-\alpha \eta} \,d\eta 
\int_{0}^{\xi-\eta}\varphi^{\alpha}(\zeta) e^{-\alpha \zeta}\,d\zeta\,.
\end{align*}

We then have
\begin{equation}\label{T3E1}
 \begin{split}
 AF_{1}(\xi) &  =\int_{0}^{\xi}\eta h^{\alpha}(\eta) e^{-\alpha \eta} - \alpha
e^{-\alpha \xi}\int_{0}^{\xi}e^{ \alpha \eta}d\eta\Big(  \int_{0}^{\eta}\zeta h^{\alpha}(\zeta) e^{-\alpha \zeta} \,d\zeta \Big) \\
&  =\int_{0}^{\xi}\eta h^{\alpha}(\eta)  e^{-\alpha \eta}d\eta-\alpha
e^{- \alpha \xi}\int_{0}^{\xi}\zeta h^{\alpha}(\zeta) e^{-\alpha \zeta}d\zeta\int_{\zeta}^{\xi}e^{\alpha \eta}d\eta\\
&  =\int_{0}^{\xi}\eta h^{\alpha}(\eta)e^{-\alpha \eta}\,d\eta - e^{-\alpha \xi}
\int_{0}^{\xi}\zeta h^{\alpha}(\zeta) e^{-\alpha \zeta}\Big(
e^{\alpha \xi}-e^{ \alpha \zeta}\Big)  d\zeta =e^{-\alpha \xi}\int_{0}^{\xi}\zeta h^{\alpha}(\zeta)\,d\zeta
\end{split}
\end{equation}
and similarly
\begin{equation}
AF_{2}(\xi)  =e^{-\alpha \xi}\int_{0}^{\xi}\eta h^{\alpha}(\eta)\,d\eta 
\int_{0}^{\xi-\eta}\zeta h^{\alpha}(\zeta)\,d\zeta
\,,\label{T3E2}%
\end{equation}%
\begin{equation}
AF_{3}(\xi)  =e^{-\alpha \xi}\int_{0}^{\xi}\eta^{2}h^{\alpha}(\eta) \varphi^{\alpha}(\xi-\eta)d\eta
- \alpha e^{-\alpha \xi}\int_0^\xi \zeta^{2}h^{\alpha}(\zeta)\,d\zeta \int_0^{\xi-\zeta}\varphi^{\alpha}(\eta)\,d\eta
\label{T3E3}%
\end{equation}
and
\begin{equation}
AF_{4}(\xi)  =e^{-\alpha \xi}\int_{0}^{\xi}\eta^{2}h^{\alpha}(\eta)  d\eta\int_{0}^{\xi-\eta}\varphi^{\alpha  }(\zeta)  d\zeta\,. \label{T3E4}%
\end{equation}

Plugging (\ref{T3E1})-(\ref{T3E4}) into (\ref{T2E6}) we obtain after some rearrangements that
$h^{\alpha}$ and $\varphi^{\alpha}$ solve \eqref{T2E3} and thus the first statement of Proposition \ref{P.VoltLikePb} follows by elementary computations.

We now compute for $\rho \in (0,1]$ the asymptotics of the solutions constructed above.

Suppose first that $\alpha<0$.

We recall that $2 B_{\infty}=\int_{0}^{\infty}\eta h_{\rho}(\eta)  d\eta= \frac{\rho}{1+\rho}<1$.
Since $0<h^{\alpha}<h$ we obtain that $b^{\alpha}$ is decreasing and $b^{\alpha}(\xi)>b^{\alpha}(\infty)>0$. Since $\alpha<0$ we have 
$\int_0^\xi \frac{q(\eta)}{\eta}e^{\alpha \eta}d\eta < \int_0^{\infty} \frac{q(\eta)}{\eta}e^{\alpha \eta}d\eta < \int_0^{\infty} \frac{q(\eta)}{\eta}d\eta =2 B_{\infty}<1$. 
Therefore, if $\alpha<0$ we have $\alpha \xi \big (  1-\int_{0}^{\xi}\frac{q(\eta)  }{\eta}e^{\alpha \eta}d\eta\big)  \rightarrow-\infty$  as $\xi\rightarrow\infty$
and thus, since $q_{\rho}(\xi)$ is bounded, we have that
$q^{\alpha}$ vanishes for some $\xi>0$ if $C_{\infty}<0$ while all the other functions remain positive.

We now consider the case $\alpha >0$. Then $h^{\alpha}$  increases exponentially and therefore $b^{\alpha}$ vanishes for some $\xi>0$.
This is obvious for all functions but $q^{\alpha}$.
To show that $q^{\alpha}$ remains positive we rewrite the formula for $q^{\alpha}$ in \eqref{T1E2}.
Integrating by parts in the last term we obtain, using also that $\frac{d}{d\xi}\big(  \frac{q(\xi)}{\xi}\big)  =-h(\xi)$, we obtain after some rearrangements
$  q^{\alpha}(\xi)   =\alpha \xi+q(\xi)  -\xi\int_{0}^{\xi}h(\eta)  \big(  e^{\alpha \eta}-1\big)  d\eta$.
On the other hand
$ b^{\alpha}(\xi)= B_{\infty} - \frac 1 2 \int_0^\xi \eta h(\eta) e^{\alpha \eta}\,d\eta$.

We now use the inequality $  e^{\alpha \eta}-1  \leq \alpha \eta e^{\alpha \eta}$ for $\eta \geq 0$.
Then
$\int_{0}^{\xi}h(\eta)  \big(  e^{\alpha \eta}{-}1\big)  d\eta\leq \alpha \int_{0}^{\xi}h(\eta) e^{\alpha \eta}d\eta$,
whence
\[ q^{\alpha}(\xi)  \geq \alpha \xi + q(\xi) - \alpha \xi \int_0^\xi h(\eta)\eta e^{\alpha \eta}d\eta 
> \alpha \xi \Big(1-\int_0^\xi h(\eta)\eta e^{\alpha \eta}d\eta\Big) \geq 2 \alpha \xi b^{\alpha}(\xi),
\]
where we used that $2B_{\infty}<1$. 
Thus, $b^{\alpha}$ vanishes before $q^{\alpha}$.

\bigskip
{\bf Acknowledgment.}
  The authors gratefully acknowledge support through the CRC 1060 {\it The mathematics of emergent effects} at the University of Bonn which is funded by
  the German Science Foundation (DFG).

\bigskip
{\small
\bibliographystyle{plain}%
\bibliography{../coagulation}%
}

\end{document}